\documentclass[12pt,oneside,english]{amsart}
\usepackage[T1]{fontenc}
\usepackage[latin9]{inputenc}
\usepackage{geometry}
\usepackage{enumitem}
\usepackage{mathrsfs}
\usepackage{amstext}
\usepackage{amsthm}
\usepackage{amssymb}
\usepackage{bm}

\usepackage[svgnames]{xcolor}
\usepackage[bookmarksnumbered=true]{hyperref} 
\hypersetup{
	colorlinks = true,
	linkcolor = Blue,
	anchorcolor = blue,
	citecolor = Green,
	filecolor = blue,
	urlcolor = FireBrick
}
\usepackage{bbm}
\usepackage{comment}

\usepackage{algorithm}
\usepackage{algpseudocode}

\usepackage{aliascnt}
\usepackage[nameinlink]{cleveref}

\makeatletter
\numberwithin{equation}{section}
\numberwithin{figure}{section}

\theoremstyle{plain}
\newtheorem{theorem}{Theorem}[section]

\theoremstyle{definition}
\newaliascnt{definition}{theorem}
\newtheorem{definition}[definition]{Definition}
\aliascntresetthe{definition}
\crefname{definition}{Definition}{Definitions}

\theoremstyle{definition}
\newaliascnt{question}{theorem}

\aliascntresetthe{question}
\crefname{question}{Question}{Questions}

\theoremstyle{definition}
\newaliascnt{remark}{theorem}
\newtheorem{remark}[remark]{Remark}
\aliascntresetthe{remark}
\crefname{remark}{Remark}{Remarks}

\theoremstyle{plain}
\newaliascnt{corollary}{theorem}
\newtheorem{corollary}[corollary]{Corollary}
\aliascntresetthe{corollary}
\crefname{corollary}{Corollary}{Corollaries}

\theoremstyle{plain}
\newaliascnt{lemma}{theorem}
\newtheorem{lemma}[lemma]{Lemma}
\aliascntresetthe{lemma}
\crefname{lemma}{Lemma}{Lemmas}

\theoremstyle{plain}
\newaliascnt{proposition}{theorem}

\aliascntresetthe{proposition}
\crefname{proposition}{Proposition}{Propositions}

\theoremstyle{definition}
\newaliascnt{example}{theorem}

\aliascntresetthe{example}
\crefname{example}{Example}{Examples}

\theoremstyle{definition}
\newaliascnt{assumption}{theorem}
\newtheorem{assumption}[assumption]{Assumption}
\aliascntresetthe{assumption}
\crefname{assumption}{Assumption}{Assumptions}

\theoremstyle{definition}
\newaliascnt{problem}{theorem}

\aliascntresetthe{problem}
\crefname{problem}{Problem}{Problems}

\newtheorem*{definition*}{Definition}
\newtheorem*{proposition*}{Proposition} 
\newtheorem*{remark*}{Remark}
\newtheorem*{example*}{Example}
\newtheorem*{problem*}{Problem}
\newtheorem*{question*}{Question}
\newtheorem*{theorem*}{Theorem}
\newtheorem*{lemma*}{Lemma}
\newtheorem*{corollary*}{Corollary}
\newtheorem*{Acknowledgments*}{Acknowledgments}

\renewenvironment{proof}[1][\proofname]{\medskip \noindent {\bfseries #1. }}{\hfill \qedsymbol\medskip}

\MakeRobust{\ref}
\newcommand{\labeltext}[2]{
\@bsphack
\csname phantomsection\endcsname
\def\@currentlabel{#1}{\label{#2}}
\@esphack
}

\usepackage{scalerel}
\usepackage[usestackEOL]{stackengine}
\def\dashint{\,\ThisStyle{\ensurestackMath{%
  \stackinset{c}{.2\LMpt}{c}{.5\LMpt}{\SavedStyle-}{\SavedStyle\phantom{\int}}}%
  \setbox0=\hbox{$\SavedStyle\int\,$}\kern-\wd0}\int}

\usepackage{enumitem}

\DeclareRobustCommand{\SkipTocEntry}[5]{}

\newcommand{\mR}{\mathbb{R}}   
\newcommand{\mN}{\mathbb{N}}   
\newcommand{\mK}{\mathbb{K}}   
\newcommand{\mS}{\mathbb{S}}   

\newcommand{\abs}[1]{\lvert #1 \rvert}  
\newcommand{\norm}[1]{\lVert #1 \rVert}  

\newcommand{\mD}{\mathscr{D}}

\newcommand{\mF}{\mathscr{F}}

\newcommand{\mJ}{\mathcal{J}}

\newcommand{\calE}{\mathcal{E}}

\newcommand{\mX}{\mathcal{X}}

\newcommand{\bfu}{\mathbf{u}}
\newcommand{\bfv}{\mathbf{v}}
\newcommand{\bfw}{\mathbf{w}}
\newcommand{\bfz}{\mathbf{z}}

\newcommand{\bfV}{\mathbf{V}}

\newcommand{\sfm}{\mathsf{m}}

\DeclareMathOperator{\supp}{supp}

\newcommand{\rmd}{\mathrm{d}}

\makeatother

\usepackage{babel}

\usepackage{orcidlink}

\begin{document}

\title[A variational model with finite carrying capacity in Euclidean space]{A variational model with finite carrying capacity for segregated multi-species populations in the whole Euclidean space}

\begin{sloppypar}

\begin{abstract} 
We study a variational problem motivated by models of multiple segregated species in a nonhomogeneous environment. The analysis combines tools from potential theory, in particular free boundary methods. 
The main novelty is that we formulate the variational problem in the whole Euclidean space while imposing a finite carrying capacity constraint, which allows us to investigate saturated regions where the species attain their maximal density. 
In addition, we provide a modeling interpretation of the constrained energy. We show that it arises as the variational limit of energies with increasingly strong crowding and interspecific competition penalties, in the sense that, regardless of their relative rates, every sequence of minimizers of the penalized energies admits a subsequence that converges strongly to a minimizer of the constrained energy. 
\end{abstract}

\author[Kow]{Pu-Zhao Kow\,\orcidlink{0000-0002-2990-3591}}
\address{Department of Mathematical Sciences, National Chengchi University, Taipei 116, Taiwan \& Kanazawa University, Kakuma, 920-1192 Kanazawa, Japan} 
\email{pzkow@g.nccu.edu.tw}

\author[Wu]{Chang-Hong Wu\,\orcidlink{0000-0001-7290-1615}}
\address{Department of Applied Mathematics, National Yang Ming Chiao Tung University, Hsinchu, 300, Taiwan}
\email{changhong@math.nctu.edu.tw}

\thanks{Kow was supported by the National Science and
Technology Council of Taiwan (NSTC 115-2628-M-004-001), and by the National Center for Theoretical Sciences of Taiwan.
Wu was supported by the National Science and
Technology Council of Taiwan (NSTC 113-2628-M-A49-004-MY4), and by the National Center for Theoretical Sciences of Taiwan.}

\maketitle

\tableofcontents{}

\subsection*{Data Availability Statement} 

No datasets were generated or analysed during the current study. 

\subsection*{Notations} 

To incorporate tools from potential theory, we adopt the notation of \cite{KSS26Multiphase}.

\section{Introduction} 

Spatial segregation of populations is a common phenomenon in ecology (see, e.g., \cite{SpatialSegregationSpecies} for recent experimental observations). In \cite{MK1980segregated,SKT1979segregated}, 
interaction-diffusion and cross-diffusion models showed how spatially heterogeneous patterns and segregation may be generated by density-dependent dispersal and competitive interactions. 
In particular, two competing species become spatially segregated as the interspecific competition rates tend to infinity, a regime known as the spatial segregation limit \cite{CDHMN2004segregated,DHMP1999segregated,HIMN01}, see also the subsequent works \cite{IMW21,IMW2023segregated,ZZLL2012segregated}. 
Closely related strong-coupling limits for variational elliptic systems can be found in \cite{CL08,NTTV10}. 

A variational theory of spatial segregation was established in \cite{CTV05SegregationProblem}.  
They developed a variational framework for
spatially segregated stationary states of an arbitrary number of competing densities on bounded domains.
These states are obtained by minimizing the sum of internal energies among all segregated configurations. Subsequent numerical investigations were carried out in \cite{AB16population,Bozorgnia2009Segregated,BA13segregated,BAKV2016PartialSegregated}. We will explain later the differences between our work and these studies. 

Let $d\ge 2$ and $m\ge 2$ be integers. 
We consider $m$ species dispersing in an unbounded heterogeneous habitat represented by $\mR^d$ and study their spatially segregated equilibrium configurations. 
Let $u_{1},\cdots,u_{m}\ge 0$ denote the population densities at each point $x\in\mR^{d}$. For each $x\in\mR^{d}$, we impose a finite carrying capacity $u_{i}(x)\le 1$ for all $i=1,\cdots,m$ and a segregated condition
\begin{equation*}
u_{i}(x)u_{j}(x)= 0 \quad \text{for all $i\neq j$.} 
\end{equation*} 
The heterogeneity of the environment is described by potential functions $V_{i}=V_{i}(x)$ for $i=1,\cdots,m$. Here, $V_{i}(x)>0$ corresponds to an unfavorable (or ``bad'') environment for $i^{\rm th}$ species, while $V_{i}(x)<0$ corresponds to a favorable (or ``good'') one. In general, the functions $V_{i}$ may differ, reflecting that the species can respond differently to the environment.  

Suppose that the species undergo diffusive motion. Then, in the spirit of \cite{KSS26Multiphase}, the equilibrium states can be described by the following minimization problem:
\begin{equation}
\begin{aligned} 
& \text{minimizing} && \mF_{\bfV}(\bfu) := \sum_{i=1}^{m}\mF_{V_{i}}(u_{i}) = \sum_{i=1}^{m}\int_{\mR^{d}} \left(\abs{\nabla u_{i}}^{2} + V_{i}(x)\abs{u_{i}}^{2}\right)\,\rmd\sfm \\ 
& \text{subject to} && \mS_{m} := \left\{ \bfu\in(H^{1}(\mR^{d}))^{m} : \begin{aligned}
& 0\le u_{i} \le 1 && \text{for $i=1,\cdots,m$} \\ 
& u_{i}\cdot u_{j}=0 && \text{for $i\neq j$}
\end{aligned} \right\} 
\end{aligned} \label{eq:main-minimizing} 
\end{equation}
where we denote $\bfu=(u_{1},\cdots,u_{m})$ and $\bfV=(V_{1},\cdots,V_{m})$. 
A vector-valued function $\bfu_{*} \in \mS_{m}$ is called a (global) minimizer of \eqref{eq:main-minimizing} in $\mS_{m}$ if $\mF_{\bfV}(\bfu_{*})\le\mF_{\bfV}(\bfu)$ for all $\bfu\in\mS_{m}$. 

The case when $m=1$ was studied in \cite{KKO26population}. The main objective of the present paper is to investigate the case $m\ge 2$. A related variational problem in one dimension ($d=1$) was analyzed in \cite{AKO24TwoStepMinimization2, AKO25TwoStepMinimization1}, where the existence of minimizers was established using the Sobolev embedding $H^{1}(\mR^{1})\subset C^{0}(\mR^{1})$ and the fact that, for each $u\in H^{1}(\mR^{1})$, there exists $a\in\mR^{1}$ such that $\abs{u(a)}=\norm{u}_{L^{\infty}(\mR)}$. These properties fail in higher dimensions, which prevents a direct extension of their arguments to the case $d \ge 2$.


We also note that \cite{AB16population,Bozorgnia2009Segregated,BA13segregated,BAKV2016PartialSegregated,CTV05SegregationProblem} consider related settings. The main distinction is that we work in the whole Euclidean space, whereas their analyses are carried out on bounded domains. Moreover, the functional $\mF_{\bfV}$ in \eqref{eq:main-minimizing} is not coercive in general (cf. \Cref{assu:existence}). Instead, we impose a finite carrying-capacity constraint $u_i\leq 1$. Consequently, the existing theories do not directly apply to our framework.


The two constraints defining $\mS_m$ can be interpreted as idealized infinite-strength limits of two distinct ecological mechanisms: rapid regulation above the local carrying capacity (represented by the constraint $u_i\leq1$), and strong interspecific competition (represented by the segregation condition $u_i u_j=0$). We now make this interpretation precise.
Let $(H_{+}^{1}(\mR^{d}))^{m}$ denote the Cartesian product of $m$ copies of 
\begin{equation*}
H_{+}^{1}(\mR^{d}) :=\left\{ u\in H^1(\mR^d) : u\geq0 \right\}, 
\end{equation*}
that is, 
\begin{equation*}
(H_{+}^{1}(\mR^{d}))^{m} := \left\{ \bfu\in(H^1(\mR^d))^m : u_{i}\geq0 \right\}.
\end{equation*}
We denote 
\begin{equation*}
a_{+} := \max\{a,0\} ,\quad a_{-} := -\min\{a,0\}. 
\end{equation*}
For each $\varepsilon,\kappa>0$, we penalize excess above the normalized carrying capacity $1$ and overlap between population densities through the energy 
\begin{equation} 
\calE_{\bfV,\varepsilon,\kappa}(\bfu)
:=\mF_{\bfV}(\bfu) +\frac{1}{\varepsilon}\sum_{i=1}^{m} \int_{\mR^d}(u_{i}-1)_{+}^{2}\,\rmd\sfm + \kappa \sum_{1\leq i<j\leq m} \int_{\mR^d}u_{i}^{2}u_{j}^{2}\, \rmd\sfm. \label{eq:energy1}
\end{equation}
We set $\calE_{\bfV,\varepsilon,\kappa}(\bfu)=+\infty$ if  $\bfu\notin (H_{+}^{1}(\mR^{d}))^{m}$ or if any of the overlap integrals is infinite. 
The second term is the Moreau-Yosida regularization of the upper-obstacle constraint $u_i\leq1$, see \cite{attouch1984variational}. Related penalization methods for obstacle problems can be found in \cite{kinderlehrer2000introduction}. The last term is a quartic overlap penalty commonly used in variational models of strong competition and phase segregation, see \cite{CL08,NTTV10}.
For finite $\varepsilon$ and $\kappa$, populations may slightly exceed the carrying capacity and different species may overlap, but both effects incur an energetic cost. Here, $\varepsilon^{-1}$ measures the strength of the crowding penalty, while $\kappa$ measures the strength of interspecific competition. In the associated $L^2$-gradient flow, $\varepsilon$ also determines the time scale of regulation above the carrying capacity.

Reaction-diffusion models of population dispersal have a long history, dating back to the seminal works \cite{Fisher1937,KPP1937}, see also \cite{Skellam1951PopulationDynamics, Aronson1985SkellamRevisit,GM1977Population}. 
In the present setting, for nonnegative initial data, the formal $L^2$-gradient flow associated with \eqref{eq:energy1}, after a constant rescaling of time, is given by 
\begin{equation}\label{eq:penalized-rd}
\partial_{t} u_{i}=\Delta u_{i}+r_{i}(x)u_{i}-\frac{1}{\varepsilon}(u_{i}-1)_{+} -\kappa u_{i}\sum_{j\ne i}u_{j}^{2}, \quad \text{for all $i=1,\cdots,m$.} 
\end{equation}
Here, $r_{i}=-V_{i}$ represents the intrinsic growth rate of the $i^{\rm th}$ species. 
Note that \eqref{eq:penalized-rd} is intended as a phenomenological relaxation of the two ecological constraints defining $\mS_m$, rather than as a canonical Lotka-Volterra model. 
In \eqref{eq:penalized-rd}, we interpret $u_{i}$ as a normalized local population density, with $u_{i}=1$ representing the local carrying capacity. When $u_{i}$ exceeds $1$, the term $\frac{1}{\varepsilon}(u_{i}-1)_{+}$ removes the excess density on the fast time scale $\varepsilon$, which may represent mortality, emigration, or unsuccessful settlement. This mechanism is inactive when the population is below the carrying capacity. 
For example, in the spatially homogeneous scalar case with constant $r>0$ and no competition, the regulated equilibrium is 
\begin{equation*} 
u_{\varepsilon}=\frac{1}{1-\varepsilon r} = 1+\varepsilon r+O(\varepsilon^{2}),\quad 0<\varepsilon r<1.
\end{equation*} 
Thus, the finite $\varepsilon$ model permits an $O(\varepsilon)$ excess above the carrying capacity, while the constraint $u_{i}\leq1$ is recovered in the limit $\varepsilon\searrow0$. 

The interaction term $\kappa u_{i}\sum_{j\ne i}u_{j}^{2}$ in \eqref{eq:penalized-rd} models symmetric, density-dependent interspecific competition. The particular power is a modeling choice dictated by the variational structure: it yields a symmetric $L^2$-gradient flow and allows the strong-competition limit to be analyzed using energy methods. 

A standard framework for studying singular limits of variational problems is provided by $\Gamma$-convergence, see, e.g., \cite{attouch1984variational,braides2002gamma,dal2012introduction}.
Our $\Gamma$-convergence result (\Cref{thm:singular-limit}) makes this interpretation rigorous: as $\varepsilon\searrow0$ and $\kappa\to\infty$, independently of their relative rates, the functional $\calE_{\bfV,\varepsilon,\kappa}$ $\Gamma$-converges, with respect to the weak topology of $(H^1(\mathbb{R}^d))^m$, to the functional $\mF_{\bfV}$ on $\mS_{m}$ and to $+\infty$
outside $\mS_{m}$. Consequently, every weak limit point of a sequence of global minimizers of $\calE_{\bfV,\varepsilon,\kappa}$ is a minimizer of \eqref{eq:main-minimizing}. Moreover, for any
sequences $\varepsilon_n\searrow0$ and $\kappa_n\rightarrow\infty$, each corresponding sequence of global minimizers admits a subsequence converging strongly in $(H^{1}(\mR^d))^m$ to such a minimizer. 
Thus, the segregated problem arises as the simultaneous fast-crowding and strong-competition limit of the penalized model.

\section{Main results}

\subsection{Regularity of minimizers in \texorpdfstring{$\mS_{m}$}{Sm}} 

As with other variational problems, we begin by studying the Euler-Lagrange equation associated with the minimizing problem \eqref{eq:main-minimizing}, analogous to \cite[Lemma~3.3]{KSS26Multiphase}. An existence result will be proved later in \Cref{thm:existence}, and a sufficient condition for the co-existence of $m$ populations will be given in \Cref{propr:m-coexist}. 
For each $j =1,\cdots,m$, an expression of the form 
\begin{equation*}
\sum_{\ell\neq j}a_{\ell}
\end{equation*}
denotes the sum over all $\ell\in\{1,\cdots,m\}\setminus\{j\}$. This convention should be clear from the context. In particular, summing over $\ell\ne j$ is not the same as summing over $j\ne\ell$, where in the latter case $\ell$ is fixed and the sum is taken over $j$. 
With this notation at hand, we can now state our first main result. 

\begin{theorem}\label{thm:1} 
Let $\bfV = (V_{1},\cdots,V_{m}) \in (L^{\infty}(\mR^{d}))^{m}$. Suppose that $\bfu_{*}=(u_{*,1},\cdots,u_{*,m})$ is a minimizer of \eqref{eq:main-minimizing} in $\mS_{m}$. Then, for each $j=1,\cdots,m$, the function $u_{*,j}$ has a continuous representative and satisfies 
\begin{equation}
(\Delta-V_{j})u_{*,j}  - \sum_{\ell\neq j}(\Delta-V_{\ell})u_{*,\ell} \le (V_{j})_{-}\chi_{K_{j}} \quad \text{in $\mD'(\mR^{d})$,}
\label{eq:Euler-Lagrange}
\end{equation}
where $K_{j}:=\{u_{*,j}=1\}$. For each $j=1,\cdots,m$, we define the open set $D_{j}:=\{u_{*,j}>0\}$. Moreover: 
\begin{enumerate} 
\renewcommand{\labelenumi}{\theenumi}
\renewcommand{\theenumi}{\rm (\roman{enumi})} 
\item \label{itm:two-phase} If $m=2$, then $\bfu_{*}$ admits a locally Lipschitz continuous representative on $\mR^{d}$. 
\item \label{itm:multi-phase} If $m\ge 3$, then $\bfu_{*}$ admits a locally Lipschitz continuous representative on 
$\mR^{d}\setminus A$, where
\[
A:=\bigcup_{1\leq i<j<k\leq m}(\partial D_{i}\cap\partial D_{j}\cap\partial D_{k}).
\] 
\end{enumerate}
\end{theorem}

With \eqref{eq:Euler-Lagrange} and the existence of a continuous representative of $\bfu_{*}$ at hand, it is straightforward to establish \ref{itm:two-phase} and \ref{itm:multi-phase}. 
From \eqref{eq:Euler-Lagrange}, it follows that 
\begin{equation*}
- \sum_{\ell\neq k}(\Delta - V_{\ell})u_{*,\ell} \le 0 \quad \text{in $\mD'(\mR^{d}\setminus \overline{D_{k}})$.} 
\end{equation*}
Fix $j\in\{1,2,\cdots,m\}$ and $k \neq j$. From the above inequality, we deduce 
\begin{equation*}
-(\Delta - V_j)u_{*,j} \le 0 \quad \text{in $\mD'\left(\mR^{d}\setminus \bigcup_{\ell\neq j}\overline{D_{\ell}}\right)$.} 
\end{equation*}
Combining this with \eqref{eq:Euler-Lagrange}, we conclude that $u_{*,j}\in W_{\mathrm{loc}}^{2,p}(\mR^{d}\setminus\bigcup_{\ell\neq j}\overline{D_{\ell}})$ for all $1\le p<\infty$ and 
\begin{equation*}
0 \le (\Delta - V_{j})u_{*,j} \le (V_{j})_{-}\chi_{K_{j}} \quad \text{a.e. in $\mR^{d}\setminus\bigcup_{\ell\neq j}\overline{D_{\ell}}$.} 
\end{equation*} 
Hence, by Sobolev embedding, it admits a $C_{\rm loc}^{1,\alpha}$ representative in $\mR^{d}\setminus\bigcup_{\ell\neq j}\overline{D_{\ell}}$ for every $0<\alpha<1$. 
Applying \cite[Lemma~7.7]{GT01Elliptic} to $\nabla u_{*,j}$ yields $\Delta u_{*,j}=0$ a.e. in $K_{j}$, then 
\begin{equation}
\Delta u_{*,j} - V_{j}u_{*,j}\chi_{\{u_{*,j}<1\}}=0 \text{ a.e. in $\mR^{d}\setminus\bigcup_{\ell\neq j}\overline{D_{\ell}}$} \quad\text{and}\quad V_{j}\le0 \text{ a.e. in $K_{j}$.} \label{eq:1-phase-EL}
\end{equation} 
This is precisely a Lewy-Stampacchia-type inequality for the one-phase setting (cf. \cite[Theorem~1.1]{KKO26population}). 

From \eqref{eq:Euler-Lagrange}, it is straightforward to check that 
\begin{equation}
-(V_{k})_{-}\chi_{K_{k}} \le (\Delta-V_{j})u_{*,j}  - (\Delta-V_{k})u_{*,k} \le (V_{j})_{-}\chi_{K_{j}} \quad \text{in $\mD'\left(\mR^{d}\setminus\bigcup_{\ell\neq j,k}\overline{D_{\ell}}\right)$.} \label{eq:2-phase-EL}
\end{equation}
This shows that $\Delta(u_{*,j}-u_{*,k})$ is bounded in $\mR^{d}\setminus\bigcup_{\ell\neq j,k}\overline{D_{\ell}}$, and therefore $u_{*,j}-u_{*,k}\in W_{\rm loc}^{2,p}(\mR^{d}\setminus\bigcup_{\ell\neq j,k}\overline{D_{\ell}})$ for all $1\le p<\infty$, and \eqref{eq:2-phase-EL} holds a.e. in $\mR^{d}\setminus\bigcup_{\ell\neq j,k}\overline{D_{\ell}}$. Hence, by Sobolev embedding, it admits a $C_{\rm loc}^{1,\alpha}$ representative in $\mR^{d}\setminus\bigcup_{\ell\neq j,k}\overline{D_{\ell}}$ for every $0<\alpha<1$. It follows immediately that  
\begin{equation}
(u_{*,j}-u_{*,k})_{+} = u_{*,j} ,\quad (u_{*,j}-u_{*,k})_{-} = u_{*,k} \quad \text{belongs to $C_{\rm loc}^{0,1}\left(\mR^{d}\setminus\bigcup_{\ell\neq j,k}\overline{D_{\ell}}\right)$.} \label{eq:Lipschitz-regularity} 
\end{equation}
This completes the proof of \ref{itm:two-phase} and \ref{itm:multi-phase}. 

However, we emphasize that the continuity of $\bfu_{*}$ on $\mR^{d}$, especially when crossing the free boundaries $\partial D_{i}\cap\partial D_{j} \cap \partial D_{k}$ for any distinct triple $(i,j,k)$, is obtained using techniques from potential analysis (see \cite{AG01Potential} for background). More precisely, we follow the approach in \cite[Corollary~3.4]{KSS26Multiphase}, which relies on general properties of subharmonic functions, such as \cite[Theorem~4.1.8]{Hoermander_book_1}, together with the Kato's inequality (see, e.g., \cite{BP04KatoInequality} or \cite[Corollary~2.3]{GS12TwoPhaseQD} for a short alternative proof). 

From \eqref{eq:2-phase-EL} we reach the Euler-Lagrange equation 
\begin{equation}
\Delta(u_{*,j}-u_{*,k}) = V_{j}u_{*,j} - V_{k}u_{*,k} \quad \text{a.e. in $\mR^{d}\setminus\left(K_{j}\cup K_{k}\cup\bigcup_{\ell\neq j,k}\overline{D_{\ell}}\right)$,} \label{eq:neighborhood}
\end{equation}
which corresponds to the two-phase problem described in \cite[Section~1.2.4]{PSU12FreeBoundary}. 
In view of \eqref{eq:Lipschitz-regularity}, a direct application of Schauder's estimate yields the following $C_{\rm loc}^{2,\alpha}$ regularity of $u_{*,j} - u_{*,k}$ in $\mR^{d}\setminus(K_{j}\cup K_{k}\cup\bigcup_{\ell\neq j,k}\overline{D_{\ell}})$, under mild assumptions on $\bfV$: 

\begin{corollary}\label{cor:C11-regularity} 
Let $\bfV = (V_{1},\cdots,V_{m}) \in (L^{\infty}(\mR^{d}))^{m}$. Suppose that $\bfu_{*}=(u_{*,1},\cdots,u_{*,m})$ is a minimizer of \eqref{eq:main-minimizing} in $\mS_{m}$. If $V_{j},V_{k}\in C_{\rm loc}^{\alpha}(U)$ for some open set $U\subset \mR^{d}\setminus(K_{j}\cup K_{k}\cup\bigcup_{\ell\neq j,k}\overline{D_{\ell}})$, then $u_{*,j}-u_{*,k}\in C_{\rm loc}^{2,\alpha}(U)$. 
\end{corollary}

If one carefully inspects the proof of \Cref{thm:1}, one sees that, for each $j = 1, \dots, m$,
\begin{equation}
\Delta \left( u_{*,j} - \sum_{\ell \ne j} u_{*,\ell} \right) \quad \text{is a signed Radon measure.} \label{eq:generalized-normal-derivative}
\end{equation}
By a formal integration by parts, the Radon measure in \eqref{eq:generalized-normal-derivative} can be interpreted as a generalized normal derivative, even though in potential theory it is typically represented by $\Delta$ rather than $\partial_\nu$, see, for example, the description of harmonic measure in terms of (classical) balayage \cite[(2.6)]{Gus04LecturesBalayage}. 
However, it does not seem possible to ensure that $\Delta u_{*,j}$ is well defined, especially across the free boundaries $\partial D_{j}$. This prevents us from obtaining an analogue of \cite[Proposition~1.7]{KKO26population} in the case $m \ge 2$. 

It should be noted that applying the same argument as in \cite[Theorem~1.6]{KKO26population} yields the following $C^{1,1}$-regularity of $u_{*,j}$ across the free boundary $\partial K_{j}$, which holds even without assuming $C^{\alpha}$ regularity of $V_{j}$ (cf. \Cref{cor:C11-regularity}): 

\begin{corollary}\label{cor:C11}
Let $\bfV = (V_{1},\cdots,V_{m}) \in (L^{\infty}(\mR^{d}))^{m}$. Suppose that $\bfu_{*}=(u_{*,1},\cdots,u_{*,m})$ is a minimizer of \eqref{eq:main-minimizing} in $\mS_{m}$. If $K_{j}\neq\emptyset$ for some $j=1,\cdots,m$, then  
\begin{equation*}
0 \le 1-u_{*,j} \le C_{1}{\rm dist}_{\{u_{*,j}=1\}}(x)^{2} \quad \text{and} \quad \abs{\nabla u_{*,j}} \le C_{2}{\rm dist}_{\{u_{*,j}=1\}}(x)
\end{equation*}
for all $x\in\mR^{d}$, where 
\begin{equation*}
C_{1} = \frac{2^{d}\norm{(V_{j})_{-}}_{L^{\infty}(\mR^{d})} + \norm{(V_{j})_{+}}_{L^{\infty}(\mR^{d})}}{2d}, \quad C_{2} = 2^{d-1}\norm{(V_{j})_{-}}_{L^{\infty}(\mR^{d})} + \frac{d\norm{V_{j}}_{L^{\infty}(\mR^{d})}}{d+1}, 
\end{equation*}
and ${\rm dist}_{\{u_{*,j}=1\}}(x)=\inf_{y\in\{u_{*,j}=1\}}\abs{x-y}$ for all $x\in\mR^{d}$. 
\end{corollary}

\subsection{Regularity of free boundaries\label{sec:regularity-FB}} 

Applying the same argument as in \cite[Theorem~1.8]{KKO26population} yields the following result concerning the geometry of the free boundary $\partial K_j$: 

\begin{corollary}\label{cor:FB}
Suppose that all assumptions in \Cref{cor:C11} hold. If $K_{j}$ is a Lipschitz domain and $V_{j}$ is $C^{\alpha}$ for some $0<\alpha<1$ near $x_0\in\partial K_{j}$ and $V_{j}(x_{0})<0$, then $u_{*,j}$ reaches the optimal regularity $C^{1,1}$ near $x_0$ and the following holds true: 
\begin{enumerate}
\item [(a)] If $d=2$ and $\partial K_{j}$ is piecewise $C^{1}$, then $\partial K_{j}$ is $C^{1}$ near $x_0$. 
\item [(b)] If $d=2$ and $K_{j}$ is convex, then $\partial K_{j}$ is $C^{1}$ near $x_0$. 
\item [(c)] If $d\ge 3$, then $x_0$ is not an edge point described in \cite[Definition~1.3]{SS25vanishingcontrast}. 
\end{enumerate}
\end{corollary}

Both \Cref{cor:FB} and \cite[Theorem~1.8]{KKO26population} are proved using free boundary methods developed in \cite{SS25vanishingcontrast}, which have also been applied to scattering problems in \cite{KSS25AnisotropicII}. The analysis is based on a blow-up argument, made possible by introducing a suitable Weiss energy functional (cf. \cite[Section~3.5]{PSU12FreeBoundary}). After classifying the blow-ups in two dimensions, the result follows from nondegeneracy properties of no-sign obstacle problems (cf. \cite{PSU12FreeBoundary}) together with a weak flatness result.
In higher dimensions, the proof relies on Federer's dimension reduction argument (cf. \cite{Weiss99FB} or \cite[Lemma~10.9]{Velichkov23FB}).

In view of \eqref{eq:neighborhood} and the $C_{\rm loc}^{2,\alpha}$ regularity of $u_{*,j}-u_{*,k}$ given in \Cref{cor:C11-regularity}, it is natural to investigate the free boundary $(\partial D_{j}\cap\partial D_{k})\setminus\bigcup_{\ell\neq j,k}\overline{D_{\ell}}$ for $j\neq k$. However, there is no a priori reason to assume that $\abs{\nabla(u_{*,j}-u_{*,k})(x_0)}\neq 0$. Consequently, the free boundary results for two-phase problems in \cite{DSFS14TwoPhaseProblemLinear} cannot be applied directly to this setting. This type of problem was first studied in \cite{Caffarelli1987FreeBoundaryI}, see also \cite{STV19TwoPhaseProblemFullyLinear} for the fully nonlinear case. 
Since the right-hand side of \eqref{eq:neighborhood} vanishes at $x_0$, the local regularity theory for two-phase free boundary problems \cite[Theorem~6.12]{PSU12FreeBoundary} is also not applicable. 
If $\abs{\nabla(u_{*,j}-u_{*,k})(x_0)}= 0$, one would need to extend the methods in \cite{SS25vanishingcontrast, KSS25AnisotropicII} to the two-phase setting, which appears to be challenging. We leave this problem for future investigation.

\subsection{Existence of minimizers in \texorpdfstring{$\mS_{m}$}{Sm}}

We now assume that $\bfV$ satisfies the following assumption. 

\begin{assumption}\label{assu:existence} 
Suppose that $\bfV=(V_{1},\cdots,V_{m})\in (L^{\infty}(\mR^{d}))^{m}$ satisfies 
\begin{equation}
\min_{\ell=1,\cdots,m}V_{\ell}(x) \ge \alpha^{2} \quad\text{for a.e. $x\in \mR^{d}\setminus B_{R}$} \label{eq:nonvanishing}
\end{equation}
and for each $j=1,\cdots,m$
\begin{equation*}
V_{j}(x)\le -\beta^{2} \quad\text{for a.e. $x\in G_{j}$}
\end{equation*}
for some $\alpha>0$, $\beta>0$, $R>0$ and nonempty bounded domains $G_{1},\cdots,G_{m}\subset B_{R}$ (``potential wells''). In addition, we assume that 
\begin{equation}
\beta^{2} > \max_{j=1,\cdots,m}\lambda^{*}(G_{j}) ,\quad \lambda^{*}(G):=\inf_{\phi\in C_{c}^{\infty}(G),\phi\not\equiv0}\frac{\norm{\nabla\phi}_{L^{2}(G)}^{2}}{\norm{\phi}_{L^{2}(G)}^{2}}. \label{eq:potential-well}
\end{equation}
\end{assumption}

It is well-known that the fundamental tone $\lambda^{*}(G)$ is positive for any bounded domain $G$. Let 
\begin{equation}
\mK = \left\{ u\in H^{1}(\mR^{d}) : 0\le u \le 1 \right\}. \label{eq:constraint-K}
\end{equation}
It is straightforward to see that $\mF_{\bfV}|_{\mK^{m}}$ (and so is $\mF_{\bfV}|_{\mS_{m}}$) is bounded below: 
\begin{equation*}
\mF_{\bfV}(\bfu) \ge \sum_{j=1}^{m}\int_{B_{R}} V_{j}\abs{u_{j}}^{2}\,\rmd\sfm \ge -\sum_{j=1}^{m}\abs{B_R}\norm{(V_{j})_{-}}_{L^{\infty}(\mR^{d})} \quad \text{for all $\bfu\in \mK^{m}$.}
\end{equation*}
Since the eigenvalue $\lambda^{*}(G_{j})$ is simple and has a positive eigenfunction $\phi_{j}^{*}\in H_{0}^{1}(G_{j})\cap C^{\infty}(G_{j})$ (cf. \cite[Theorem~8.38]{GT01Elliptic}), normalized such that its zero extension $\chi_{G_{j}}\phi_{j}^{*}\in H^{1}(\mR^{d})$ satisfies $0\le \chi_{G_{j}}\phi_{j}^{*} \le 1$, then we see that $(0,\cdots,0,\chi_{G_{j}}\phi_{j}^{*},0,\cdots,0) \in \mS_{m}$ and 
\begin{equation}
\begin{aligned}
& \mF_{\bfV}(0,\cdots,0,\chi_{G_{j}}\phi_{j}^{*},0,\cdots,0) = \int_{G_{j}}(\lambda^{*}(G_j)+V_{j})\abs{\phi_{j}^{*}}^{2}\,\rmd\sfm \\
&\quad \le (\lambda^{*}(G_{j})-\beta^{2})\int_{G_{j}}\abs{\phi_{j}^{*}}^{2}\,\rmd\sfm < 0. 
\end{aligned} \label{eq:nontrivial}
\end{equation} 
This shows that any minimizer $\bfu_{*}$ of \eqref{eq:main-minimizing} in $\mS_m$ must satisfy $\abs{\bfu_{*}}\not\equiv 0$. 

We now establish the existence of minimizer, stated below: 

\begin{theorem}\label{thm:existence} 
We assume that $\bfV$ satisfies \Cref{assu:existence}. Then there exists a minimizer $\bfu_{*}$ of \eqref{eq:main-minimizing} over $\mS_{m}$ such that 
\begin{equation}
u_{*,j}(x) \le v_{*,j}(x) \quad \text{for all $x\in\mR^{d}$,} \label{eq:exponential-decay1}
\end{equation}
where $v_{*,j}$ is the smallest minimizer of $\mF_{V_{j}}$ in $\mK$. 
\end{theorem}


\begin{remark}\label{rem:exponential-decay2} 
Since there exists a positive constant $C=C(R,d,\alpha)$ such that $v_{*,j}(x) \le C\abs{x}^{-\frac{d-1}{2}}e^{-\alpha\abs{x}}$ for all $x\in\mR^{d}\setminus\{0\}$, the minimizer in \Cref{thm:existence} satisfies $\norm{\bfu_{*}}_{L^\infty(\mR^d)}=1$, with the supremum attained\footnote{Otherwise, $\norm{\bfu_{*}}_{L^{\infty}(\mR^{d})}<1$, and hence $\mF_{\bfV}\left(\frac{\bfu_{*}}{\norm{\bfu_{*}}_{L^{\infty}(\mR^{d})}}\right) = \frac{1}{\norm{\bfu_{*}}_{L^{\infty}(\mR^{d})}^{2}}\overbrace{\mF_{\bfV}(\bfu_{*})}^{<\,0} < \mF_{\bfV}(\bfu_{*})$, which is a contradiction.}. In particular, the set $\{u_{*,j}=1\}$ is compact. 
\end{remark} 

\subsection{Co-existence of segregated species}

We now show that these segregated species can coexist, at least when the green oases for each species are sufficiently far apart, as stated in the following theorem. 

\begin{theorem}\label{propr:m-coexist} 
For each $j\in\{1,\cdots,m\}$, let $\breve{V}_{j}\in L^{\infty}(\mR^{d})$ satisfy
\begin{equation*}
\breve{V}_{j}(x)\ge\alpha_{j}^{2} \quad \text{for a.e. $x\in\mR^{d}\setminus B_{R_{j}}(0)$}
\end{equation*}
and 
\begin{equation*}
\breve{V}_{j}(x) \le -\beta_{j}^{2} < - \lambda^{*}(\breve{G}_{j}) \quad \text{for a.e. $x\in \breve{G}_{j}$} 
\end{equation*}
for some $\alpha_{j},\beta_{j},R_{j}>0$ and nonempty bounded domains $\breve{G}_{j}\subset B_{R_{j}}(0)$. Then there exist constants $r_{1},\cdots,r_{m}>0$, depending only on $d$ and $\breve{V}_{j}$, such that, if 
\begin{equation}
\abs{x^{(j)}-x^{(k)}} > r_{j}+r_{k} \quad \text{for all $j\neq k$,} \label{eq:mutually-far-away}
\end{equation}
then every minimizer $\bfu_{*}$ of \eqref{eq:main-minimizing} in $\mS_{m}$, with 
\begin{equation*}
V_{j}(x) = \breve{V}_{j}(x-x^{(j)}) \quad \text{for all $x\in\mR^{d}$,} 
\end{equation*}
satisfies 
\begin{equation*}
u_{*,j}\not\equiv 0 \quad \text{for all $j\in\{1,\cdots,m\}$.}
\end{equation*} 
\end{theorem}

\subsection{Variational penalization and singular limit}

We first establish the existence of minimizers for the penalized problem \eqref{eq:energy1} and study some of their basic properties.

\begin{theorem}\label{thm:well-posed} 
Assume that $\bfV\in (L^{\infty}(\mR^{d}))^m$ satisfies \eqref{eq:nonvanishing} for some $\alpha>0$ and $R>0$, and define
\begin{equation*}
\norm{\bfV_{-}}_{L^{\infty}(\mR^{d})}:=\max_{1\leq i\leq m}\norm{(V_{i})_{-}}_{L^\infty(\mR^{d})}. 
\end{equation*}
Then the following holds: 
\begin{enumerate} 
\renewcommand{\labelenumi}{\theenumi}
\renewcommand{\theenumi}{\rm (\roman{enumi})}
\item \label{itm:singular-limit-a} For any given $\overline{\varepsilon}\in(0,(4\norm{\bfV_{-}}_{L^{\infty}(\mR^{d})})^{-1})$, the sublevel set\footnote{It is nonempty, since $\mathcal{E}_{\bfV,\varepsilon,\kappa}(\mathbf{0})=0$.} 
\begin{equation*}
\left\{\bfu\in (H_{+}^{1}(\mR^{d}))^{m}: \calE_{\bfV,\varepsilon,\kappa}(\bfu)\leq 0\right\}
\end{equation*}
is bounded in $(H^1(\mathbb R^d))^m$, uniformly with respect to $\varepsilon\in(0,\bar\varepsilon)$ and $\kappa>0$. 
In particular, for any fixed $\varepsilon,\kappa>0$
with $\varepsilon \norm{\bfV_{-}}_{L^{\infty}(\mR^{d})} \le 1/4$,
$\calE_{\bfV,\varepsilon,\kappa}$ admits a global minimizer
over $(H_{+}^{1}(\mR^{d}))^{m}$. 
\item \label{itm:singular-limit-b} 
For any given $\varepsilon,\kappa>0$
with $\varepsilon \norm{\bfV_{-}}_{L^{\infty}(\mR^{d})} \le 1/4$,
every global minimizer of $\calE_{\bfV,\varepsilon,\kappa}$ over $(H_{+}^{1}(\mR^{d}))^{m}$ belongs to $W_{\rm loc}^{2,p}(\mR^{d})$ for all $p<\infty$, and hence admits a representative in $C_{\rm loc}^{1}(\mR^{d})$, satisfying 
\begin{equation}
\Delta u_{(\varepsilon,\kappa),i} - V_{i}u_{(\varepsilon,\kappa),i} = \frac{1}{\varepsilon}(u_{(\varepsilon,\kappa),i}-1)_{+} +\kappa u_{(\varepsilon,\kappa),i}\sum_{j\neq i}u_{(\varepsilon,\kappa),j}^{2} \quad \text{a.e. in $\mR^{d}$.} \label{eq:EL}
\end{equation}
Moreover, there exists a constant $C_{R,d,\alpha}>0$, which is independent of $\varepsilon$ and $\kappa$, such that every global minimizer $\bfu_{(\varepsilon,\kappa)}$ of $\calE_{\bfV,\varepsilon,\kappa}$ over $(H_{+}^{1}(\mR^{d}))^{m}$ satisfies 
\begin{equation}
\max_{i=1,\cdots,m} u_{(\varepsilon,\kappa),i}(x) \le \frac{1}{1-\varepsilon\norm{\bfV_{-}}_{L^{\infty}(\mR^{d})}}\cdot\min\Big\{1,\, C_{R,d,\alpha}(1+\abs{x})^{-\frac{d-1}{2}}e^{-\alpha\abs{x}}\Big\} \label{eq:pointwise-bound}
\end{equation} 
for all $x\in\mR^{d}$. 
\end{enumerate} 
\end{theorem}

\begin{remark}\label{rem:even-extension}
Compared with \Cref{thm:1}, at first glance, the constraint $u_{i}\ge0$ for all $i=1,\cdots,m$ in $(H_{+}^{1}(\mR^{d}))^{m}$ might appear to introduce a free boundary term into the Euler-Lagrange equation. However, a minimizer over $(H_{+}^{1}(\mR^{d}))^{m}$ is also a minimizer of the even extension 
\begin{equation*}
\tilde{\calE}_{\bfV,\varepsilon,\kappa}(\bfu) := \calE_{\bfV,\varepsilon,\kappa}(\abs{u_{1}},\cdots,\abs{u_{m}}) 
\end{equation*}
over $(H^{1}(\mR^{d}))^{m}$. Moreover, since the mapping $s\mapsto(\abs{s}-1)_{+}^{2}$ is $C^{1,1}$, no free boundary term appears in the Euler-Lagrange equation \eqref{eq:EL}. 
\end{remark}

Define 
\begin{equation*} 
\calE_{\bfV,0,+\infty}(\bfu) :=
\left\{\begin{aligned}
& \mF_{\bfV}(\bfu), && \mathbf u\in\mS_m, \\ 
& +\infty, && \bfu\notin\mS_{m}. 
\end{aligned}\right. 
\end{equation*} 
We now show that the minimization problem \eqref{eq:main-minimizing} arises as the variational limit of energies with increasingly strong crowding and interspecific competition penalties, in the following sense:

\begin{definition}\label{def:Gamma-convergence}
For any pair of sequences $\varepsilon_n\to 0^+$ and $\kappa_n\to+\infty$, we say that \emph{$\calE_{\bfV,\varepsilon_n,\kappa_n}$ $\Gamma$-converges to $\calE_{\bfV,0,+\infty}$}, denoted by $\calE_{\bfV,\varepsilon_{n},\kappa_{n}}\overset{\Gamma}{\rightarrow}\calE_{\bfV,0,+\infty}$, with respect to the weak topology of $(H_{+}^1(\mR^d))^m$, if the following two properties hold: 
\begin{enumerate} 
\renewcommand{\labelenumi}{\theenumi}
\renewcommand{\theenumi}{\rm (\alph{enumi})} 
\item \label{eq:Gamma-1}  Given any sequence $\left\{\bfw^{(n)}\right\}_{n\in\mN}\subset (H_{+}^1(\mR^d))^m$ such that $\bfw^{(n)}\rightarrow\bfw$ weakly in $(H^{1}(\mR^{d}))^{m}$, one has 
\begin{equation*}
\calE_{\bfV,0,+\infty}(\bfw) \le \liminf_{n\rightarrow\infty}\calE_{\bfV,\varepsilon_{n},\kappa_{n}}(\bfw^{(n)}). 
\end{equation*}
\item \label{eq:Gamma-2}  Given any $\bfv\in(H_{+}^1(\mR^d))^m$, there is a sequence $\left\{\bfv^{(n)}\right\}_{n\in\mN}\subset (H_{+}^1(\mR^d))^m$ such that 
$\bfv^{(n)}\to\bfv$ weakly in $(H^1(\mR^d))^m$ and 
\begin{equation*}
\calE_{\bfV,0,+\infty}(\bfv) \ge \limsup_{n\rightarrow\infty}\calE_{\bfV,\varepsilon_{n},\kappa_{n}}(\bfv^{(n)}). 
\end{equation*}
\end{enumerate} 
\end{definition}


\begin{remark}\label{rem:property-b}
Since $\calE_{\bfV,0,+\infty}(\bfv)=+\infty$ for all $\bfv\notin\mS_{m}$, property \ref{eq:Gamma-2} follows by taking the trivial sequence $\bfv^{(n)}=\bfv$ for all $n\in\mN$, since the penalties disappear on $\mS_{m}$.  
\end{remark}

\begin{remark}\label{rem:property-a-hint}
If property \ref{eq:Gamma-1} holds, then, whenever $\bfw^{(n)}\rightarrow\bfw\in (H_{+}^1(\mR^d))^m$ weakly in $(H^{1}(\mR^{d}))^{m}$ and $\liminf_{n\rightarrow\infty}\calE_{\bfV,\varepsilon_{n},\kappa_{n}}(\bfw^{(n)})<+\infty$, we necessarily have $\bfw\in\mS_{m}$. This observation suggests that we should establish, in a suitable sense (cf. \eqref{eq:hint-done}),  
\begin{equation*}
(w_{i}^{(n)}-1)_{+}\rightarrow0 ,\quad w_{i}^{(n)}w_{j}^{(n)} \rightarrow 0. 
\end{equation*}
\end{remark}

We now show, not only that the above $\Gamma$-convergence holds, but also that the resulting variational limit enjoys a strong additional property:

\begin{theorem}\label{thm:singular-limit}
Assume that all the hypotheses of \Cref{thm:well-posed} hold. Given any sequences $\varepsilon_n\to 0^+$ and $\kappa_n\to+\infty$, the following holds: 
\begin{enumerate} 
\renewcommand{\labelenumi}{\theenumi}
\renewcommand{\theenumi}{\rm (\roman{enumi})} 
\item \label{itm:singular-limit-c} 
$\calE_{\bfV,\varepsilon_n,\kappa_n}$ $\Gamma$-converges to $\calE_{\bfV,0,+\infty}$ with respect to the weak topology of $(H_{+}^1(\mR^d))^m$ in the sense of \Cref{def:Gamma-convergence}. 
\item \label{itm:singular-limit-d} Let $\left\{\bfu^{(n)}\right\}_{n\in\mN}$ be any sequence of global minimizers of $\calE_{\bfV,\varepsilon_n,\kappa_n}$ over $(H_{+}^{1}(\mR^{d}))^{m}$. Then there exist a subsequence, still denoted by $\left\{\bfu^{(n)}\right\}_{n\in\mN}$, and a global minimizer $\bfu_*$ of $\mF_{\bfV}$ over $\mS_m$ such that
\begin{equation} 
\min_{(H_{+}^{1}(\mR^{d}))^{m}}\calE_{\bfV,\varepsilon_n,\kappa_n}
\rightarrow \min_{\mS_m}\mF_{\bfV},\quad \text{as $n\rightarrow\infty$.}  \label{eq:minimizing-sequence}
\end{equation} 
and $\bfu^{(n)}\to \bfu_{*}$ strongly in $(H^{1}(\mR^d))^{m}$. 
\end{enumerate}
\end{theorem}

\begin{remark*} 
By the proof of \Cref{thm:singular-limit}\ref{itm:singular-limit-a}, for each fixed $\kappa>0$ and every sequence $\varepsilon_{n}\to0^{+}$, one can show that $\calE_{\bfV,\varepsilon_{n},\kappa}\overset{\Gamma}{\rightarrow}\calE_{\bfV,0,\kappa}$ with respect to the weak topology of $(H_{+}^{1}(\mR^{d}))^{m}$, where 
\begin{equation*} 
\calE_{\bfV,0,\kappa}(\bfu)
=\left\{\begin{aligned}
& \mF_{\bfV}(\bfu)
+ \kappa \sum_{i<j}\int_{\mR^d}u_{i}^{2}u_{j}^{2}\,\rmd\sfm, && \bfu\in\mathcal{A}, \\ 
& +\infty, && \bfu\notin\mathcal{A}. 
\end{aligned}\right. 
\end{equation*} 
and $\mathcal{A}:= \left\{\bfu\in(H_{+}^{1}(\mR^{d}))^{m} : u_i\leq1\ \text{for all $i$} \right\}$. 
The only modification to the proof is the application of Fatou's lemma to the penalty term $\kappa \sum_{i<j}\int_{\mR^d}u_{i}^{2}u_{j}^{2}\,\rmd\sfm$. 
Similarly, for each fixed $\varepsilon>0$ and every sequence $\kappa_{n}\to+\infty$, one also can show that $\calE_{\bfV,\varepsilon,\kappa_{n}}\overset{\Gamma}{\rightarrow}\calE_{\bfV,\varepsilon,+\infty}$ with respect to the weak topology of $(H_{+}^{1}(\mR^{d}))^{m}$, where 
\begin{equation*} 
\mathcal E_{\bfV,\varepsilon,+\infty}(\mathbf u)
=
\left\{\begin{aligned}
& \mF_{\bfV}(\bfu)
+\frac{1}{\varepsilon}\sum_{i}\int_{\mR^d}(u_{i}-1)_{+}^{2}\,\rmd\sfm, && \bfu\in\mathcal{G}, \\ 
& +\infty, && \bfu\notin\mathcal{G},  
\end{aligned}\right. 
\end{equation*} 
and $\mathcal{G}:= \left\{ \bfu\in(H_{+}^{1}(\mR^{d}))^{m} : u_{i}u_{j}=0\ \text{for $i\ne j$} \right\}$. 
Applying the same argument once more to these intermediate functionals yields $\calE_{\bfV,0,\kappa_{n}}
\overset{\Gamma}{\rightarrow}\calE_{\bfV,0,+\infty}$ and $\calE_{\bfV,\varepsilon_{n},+\infty}
\overset{\Gamma}{\rightarrow}\calE_{\bfV,0,+\infty}$ with respect to the weak topology of $(H_{+}^{1}(\mR^{d}))^{m}$. 
Thus, both successive limiting procedures lead to the constrained problem \eqref{eq:main-minimizing}. In contrast, \Cref{thm:singular-limit}\ref{itm:singular-limit-b} shows that the same limiting problem is obtained when $\varepsilon\to0^{+}$ and $\kappa\to+\infty$ simultaneously, regardless of their relative rates. 
\end{remark*}

Finally, we state the following corollary, which is an immediate consequence of \Cref{propr:m-coexist,thm:singular-limit}\ref{itm:singular-limit-d}, with $V_{j}=\breve{V}_{j}(\cdot- x^{(j)})$, $\alpha=\min_{j}\alpha_{j}$ and $R=\max_{j}(\abs{x^{(j)}}+R_{j})$, and demonstrates the stability of the coexistence result for the limiting segregated problem under sufficiently strong but finite crowding and interspecific competition.

\begin{corollary}\label{cor:coexistence}
Suppose that all assumptions in \Cref{propr:m-coexist} hold. Then there exist constants $\varepsilon_0,\kappa_0,\delta>0$
such that every global minimizer $\bfu_{(\varepsilon,\kappa)}$
of $\calE_{\bfV,\varepsilon,\kappa}$ over $(H_{+}^{1}(\mR^{d}))^{m}$ satisfies
\begin{equation*} 
\min_{1\leq i\leq m}\norm{u_{(\varepsilon,\kappa),i}}_{L^2(\mR^d)}
\geq \delta
\quad\text{for all $0<\varepsilon<\varepsilon_0$ and $\kappa>\kappa_0$}.
\end{equation*} 
Here, $\delta$ is independent of $\varepsilon$ and $\kappa$. 
\end{corollary}

\subsection{An alternative formulation for the case \texorpdfstring{$m=2$}{m=2}}

We observe that the trivial solution $\bfu_{*}\equiv 0$ satisfies the inequality \eqref{eq:Euler-Lagrange} in \Cref{thm:1}. Therefore, the gradient descent method in \cite{AB16population,Bozorgnia2009Segregated,BA13segregated,BAKV2016PartialSegregated} generally does not work in our setting.  
When $m\ge 3$, the minimization problem \eqref{eq:main-minimizing} does not possess a global Euler-Lagrange equation (see also \cite{KSS26Multiphase}) that can be naturally implemented numerically. We also point out that the free boundary $\partial D_i \cap \partial D_j \cap \partial D_k$ for any distinct triple $(i,j,k)$ can be quite complicated.  
Indeed, even in the planar case ($d=2$), there exist configurations of three (in fact, countably many) pairwise disjoint domains sharing the same boundary (so-called Wada lakes) in which case such a definition would fail to provide any effective boundary condition. 

As discussed above, it is natural to ask whether there exists an alternative formulation that does not explicitly depend on the free boundaries $\partial D_{i}\cap\partial D_{j}$. This can be achieved in particular in the case $m = 2$. In this case, \eqref{eq:2-phase-EL} reads 
\begin{equation}
-(V_{2})_{-}\chi_{K_{2}} \le \Delta(u_{*,1}-u_{*,2}) - V_{1}u_{*,1} + V_{2}u_{*,2} \le (V_{1})_{-}\chi_{K_{1}} \quad \text{a.e. in $\mR^{d}$.} \label{eq:eq:2-phase-EL-m2}
\end{equation}
By writing $w_{*}:=u_{*,1}-u_{*,2} \in H^{1}(\mR^{d})$, since $u_{*,1}(x)u_{*,2}(x)=0$ for all $x\in\mR^{d}$, we see that 
\begin{equation}
(w_{*})_{+}=u_{*,1} ,\quad (w_{*})_{-}=u_{*,2}, \label{eq:functions-vu}
\end{equation}
and \eqref{eq:eq:2-phase-EL-m2} can be expressed in terms of single variable: 
\begin{subequations}\label{eq:2D-EL}
\begin{equation}
-(V_{2})_{-}\chi_{\{w_{*}=-1\}} \le \Delta w_{*} - V_{1}(w_{*})_{+} + V_{2}(w_{*})_{-} \le (V_{1})_{-}\chi_{\{w_{*}=1\}} \quad \text{a.e. in $\mR^{d}$.} \label{eq:2D-EL-a}
\end{equation}
In particular, if we consider the 2-sided non-contact set $\Omega_{*}:=\{-1<w_{*}<1\}$, then we reach the nonlinear equation 
\begin{equation}
\Delta w_{*} - V_{1}(w_{*})_{+} + V_{2}(w_{*})_{-} = 0 \text{ a.e. in $\Omega_{*}$} ,\quad \nabla w_{*}=0 \text{ on $\partial\Omega_{*}$} \label{eq:2D-EL-b}
\end{equation}
\end{subequations} 
since $w_{*}\in C_{\rm loc}^{1}$ near the free boundary $\partial\Omega_{*}$. 
This suggests that we consider the following minimization problem: 
\begin{equation}
\begin{aligned} 
& \text{minimizing} && \mJ_{\bfV}(w) := \int_{\mR^{d}} \left(\abs{\nabla w}^{2} + V_{1}(x)\abs{w_{+}}^{2} + V_{2}(x)\abs{w_{-}}^{2} \right)\,\rmd\sfm \\ 
& \text{subject to} && \mK_{[-1,1]} := \left\{ w\in H^{1}(\mR^{d}) : -1\le w\le 1 \right\}. 
\end{aligned} \label{eq:main-minimizing-2-phase} 
\end{equation}
By adapting the strategy of \cite[Theorem~2]{AS16MultiPhaseQD}, one can readily verify the equivalence between the minimization problems \eqref{eq:main-minimizing} and \eqref{eq:main-minimizing-2-phase}, as stated in the following lemma (the proof is omitted). 

\begin{lemma}
Let $V_{1},V_{2}\in L^{\infty}(\mR^{d})$. If $(u_{*,1},u_{*,2})$ is a minimizer of \eqref{eq:main-minimizing} (with $m=2$) in $\mS_{2}$, then $w_{*}:=u_{*,1}-u_{*,2}$ is a minimizer of \eqref{eq:main-minimizing-2-phase} in $\mK_{[-1,1]}$. Conversely, if $w_{*}$ is a minimizer of \eqref{eq:main-minimizing-2-phase} in $\mK_{[-1,1]}$, then $((w_{*})_{+},(w_{*})_{-})$ is a minimizer of \eqref{eq:main-minimizing} (with $m=2$) in $\mS_{2}$. 
\end{lemma}

The advantage of the formulation in \eqref{eq:main-minimizing-2-phase} is that it avoids explicit dependence on the free boundary $\partial D_1 \cap \partial D_2$, whose geometry may be highly intricate, as discussed in \Cref{sec:regularity-FB}. In particular, our analysis does not rule out the possibility of pathological configurations such as space-filling curves. 
Instead, this formulation involves only the free boundaries $\partial\{w_{*}=1\}$ and $\partial\{w_{*}=-1\}$, each corresponding to a one-phase problem, which is significantly more tractable from the perspective of free boundary methods (cf. \Cref{cor:FB}). In addition, by applying the unique continuation property from sets of positive measure (see, e.g., \cite{Regbaoui2001MUCP} for $d\ge 3$ and \cite{GL93MUCP} for $d=2$) to \eqref{eq:2D-EL-b}, we obtain the following corollary, which roughly states that the two species together occupy almost the entire domain: 

\begin{corollary}\label{cor:positivity}
Let $V_{1},V_{2}\in L^{\infty}(\mR^{d})$. If $w_{*}$ is a minimizer of \eqref{eq:main-minimizing-2-phase} in $\mK_{[-1,1]}$, then either $w_{*}\equiv 0$ or $\sfm(\{w_{*}=0\})=0$. 
\end{corollary}

\section{Conclusion} 

In this paper, we use methods from potential theory, particularly free boundary analysis, to study the equilibrium configurations of multiple competing species in an unbounded heterogeneous habitat. The main novelty lies in providing a rigorous analysis in the multiphase, higher-dimensional setting $m\ge2$ and $d\ge2$, where the one-dimensional arguments used in earlier studies are no longer applicable.

We further introduce a family of penalized energies that allows population densities to exceed the normalized carrying capacity and distinct species to overlap for finite parameter values. We prove that, as the crowding parameter $\varepsilon\to0^{+}$ and the competition parameter $\kappa\to+\infty$, independently of their relative rates, these energies $\Gamma$-converge, with respect to the weak topology of $(H^1(\mR^d))^m$, to the constrained segregation energy. Moreover, any corresponding sequence of global minimizers admits a subsequence that converges strongly in $(H^1(\mR^d))^m$ to a global minimizer of the constrained problem, and the corresponding minimum energy values also converge. This result provides a modeling interpretation of the two constraints defining $\mS_m$: the upper bound $u_i\le1$ arises from rapid local crowding regulation, while the segregation condition $u_{i}u_{j}=0$ results from strong interspecific competition.

Finally, we emphasize that even when $m=2$, the interface between the two species can exhibit a highly nontrivial structure. In fact, at this stage, we cannot exclude configurations reminiscent of space-filling curves, reflecting the inherently complex nature of the free boundary. 
Even in the three-species case, many questions remain open regarding the geometry of interfaces where multiple species interact. In particular, we cannot exclude configurations reminiscent of Wada lakes. We hope that further developments in free boundary techniques will shed new light on these problems. 
We also point out that numerical simulation of nonlinear PDEs, particularly those involving free boundary terms, is challenging. We leave this aspect for future research.




\appendix 
\crefalias{section}{appendix}
\section{Proof of main results}  
 
We begin with the proof of \Cref{thm:1}. 
 
\begin{proof}[Proof of \Cref{thm:1}] 
The proof is divided into several parts. 

\medskip 

\noindent \emph{First, we prove  \eqref{eq:Euler-Lagrange}}. Let $\psi\in C_{c}^{\infty}(\mR^{d})$ be non-negative. For each $\delta>0$, define $\bfz=(z_{1},\cdots,z_{m})$ by 
\begin{equation*}
\left\{\begin{aligned}
z_{j} &:= \min\left\{\max\left\{u_{*,j}-\sum_{\ell\neq j}u_{*,\ell}+\delta\psi,0\right\},1\right\}, \\ 
z_{\ell} &:= \max\left\{u_{*,\ell} - u_{*,j} - \delta\psi,0\right\}, && \text{for all $\ell\neq j$.} 
\end{aligned}\right. 
\end{equation*}
Observe that 
\begin{equation*}
\Omega_{\delta}:=\{z_{j}>0\} \subset \left\{\sum_{\ell\neq j}u_{*,\ell}-u_{*,j}-\delta\psi<0\right\}\subset\bigcap_{\ell\neq j}\{z_{\ell}=0\}, \quad \mR^{d}\setminus\Omega_{\delta} \subset \{u_{*,j}=0\}. 
\end{equation*}
Therefore, $\bfz\in\mS_{m}$ and $0 \le \mF_{\bfV}(\bfz) - \mF_{\bfV}(\bfu_{*}) = I_{1} + I_{2}$ with 
\begin{equation*}
\begin{aligned}
I_{1} &= \sum_{\ell=1}^{m}\int_{\mR^{d}} \left( \abs{\nabla z_{\ell}}^{2} - \abs{\nabla u_{*,\ell}}^{2} \right)\,\rmd\sfm, \\ 
I_{2} &= \sum_{\ell=1}^{m}\int_{\mR^{d}} V_{\ell} \left( \abs{z_{\ell}}^{2} - \abs{u_{*,\ell}}^{2} \right)\,\rmd\sfm. 
\end{aligned}
\end{equation*}
Following the computations in \cite[Lemma~3.3]{KSS26Multiphase}, we reach 
\begin{equation*}
\begin{aligned}
I_{1} &\leq  2\delta\int_{\mR^{d}} \nabla\psi\cdot\nabla u_{*,j}\,\rmd\sfm - 2\delta\int_{\mR^{d}} \nabla\psi\cdot\nabla\left(\sum_{\ell\neq j}u_{*,\ell}\right)\,\rmd\sfm \\ 
& \quad + \delta^{2}\int_{\Omega_{\delta}}\abs{\nabla\psi}^{2}\,\rmd\sfm + (m-1)\delta^{2}\int_{\mR^{d}\setminus\Omega_{\delta}} \abs{\nabla\psi}^{2}\,\rmd\sfm. 
\end{aligned}
\end{equation*}
We proceed to compute $I_{2}$. Since $u_{*,\ell}\cdot u_{*,j}=0$ for all $\ell\neq j$, then 
\begin{equation*}
z_{\ell}-u_{*,\ell} = \left\{\begin{aligned}
& -\delta\psi && \text{if $u_{*,\ell}\ge\delta\psi$} \\ 
& -u_{*,\ell} && \text{if $u_{*,\ell}<\delta\psi$} 
\end{aligned}\right. \quad \text{for all $\ell\neq j$}
\end{equation*}
and 
\begin{equation*}
z_{j}-u_{*,j}=\left\{\begin{aligned}
& 1-u_{*,j}\le\delta\psi && \text{if $u_{*,j}\ge 1-\delta\psi$,} \\
& \delta\psi && \text{if $0<u_{*,j}<1-\delta\psi$,} \\ 
& \delta\psi && \text{if $u_{*,j}=0$ and $\sum_{\ell\neq j}u_{*,\ell}=0$,} \\ 
& \delta\psi - \sum_{\ell\neq j}u_{*,\ell} && \text{if $0<\sum_{\ell\neq j}u_{*,\ell}<\delta\psi$,} \\ 
& 0 && \text{if $\sum_{\ell\neq j}u_{*,\ell}\ge\delta\psi$,}
\end{aligned}\right. 
\end{equation*}
for all sufficiently small $\delta>0$. Hence, we obtain 
\begin{equation*}
\begin{aligned}
I_{2} &= \sum_{\ell\neq j}\int_{\mR^{d}}V_{\ell}(z_{\ell}-u_{*,\ell})(z_{\ell}+u_{*,\ell})\,\rmd\sfm + \int_{\mR^{d}}V_{j}(z_{j}-u_{*,j})(z_{j}+u_{*,j})\,\rmd\sfm \\ 
&\le -2\delta\sum_{\ell\neq j}\int_{\{u_{*,\ell}\ge\delta\psi\}}V_{\ell}u_{*,\ell}\psi\,\rmd\sfm + 2\delta\int_{\{0<u_{*,j}<1-\delta\psi\}} V_{j}u_{*,j}\psi\,\rmd\sfm \\
&\quad +2\delta\int_{\{u_{*,j}\ge1-\delta\psi\}}(V_{j})_{+}u_{*,j}\psi\,\rmd\sfm \\ 
&\quad + \delta^{2}\sum_{\ell\neq j}\int_{\{u_{*,\ell}\ge\delta\psi\}}V_{\ell}\abs{\psi}^{2}\,\rmd\sfm - \overbrace{\sum_{\ell\neq j}\int_{\{u_{*,\ell}<\delta\psi\}}V_{\ell}(u_{*,\ell})^{2}\,\rmd\sfm}^{=\,O(\delta^{2})} \\ 
&\quad + \delta^{2}\int_{\{u_{*,j}>0\}} V_{j}\abs{\psi}^{2}\,\rmd\sfm + \delta^{2}\int_{\{u_{*,j}=\sum_{\ell\neq j}u_{*,\ell}=0\}}V_{j}\abs{\psi}^{2}\,\rmd\sfm \\ 
&\quad + \overbrace{\int_{\{0<\sum_{\ell\neq j}u_{*,\ell}<\delta\psi\}}V_{j}\left(\delta\psi-\sum_{\ell\neq j}u_{*,\ell}\right)^{2}\,\rmd\sfm}^{=\,O(\delta^{2})}
\end{aligned}
\end{equation*} 
Consequently, we have 
\begin{equation*}
\begin{aligned}
0 &\le \limsup_{\delta\rightarrow 0_{+}}\frac{1}{2\delta}(I_{1}+I_{2}) \\ 
&\le \int_{\mR^{d}} \nabla\left( u_{*,j} - \sum_{\ell\neq j}u_{*,\ell} \right)\cdot\nabla\psi\,\rmd\sfm \\
&\quad + \int_{\mR^{d}} \left( \overbrace{V_{j}\chi_{\{u_{*,j}<1\}}u_{*,j} + (V_{j})_{+}\chi_{\{u_{*,j}=1\}}}^{=V_{j}u_{*,j}+(V_{j})_{-}\chi_{\{u_{*,j}=1\}}} -\sum_{\ell\neq j}V_{\ell}u_{*,\ell}\right)\psi\,\rmd\sfm 
\end{aligned}
\end{equation*}
which concludes \eqref{eq:Euler-Lagrange}. 

From \eqref{eq:Euler-Lagrange}, we obtain 
\begin{equation*}
\Delta\left( u_{*,j} - \sum_{\ell\neq j}u_{*,\ell} \right) \le C_{0}:= \norm{(V_{j})_{+}}_{L^{\infty}(\mR^{d})} + \sum_{\ell\neq j} \norm{(V_{\ell})_{-}}_{L^{\infty}(\mR^{d})} \quad \text{in $\mD'(\mR^{d})$}
\end{equation*}
Fix $R>0$ and consider the Newtonian potential $U^{C_{0}\chi_{B_{R}}}$ (cf. \cite[(1.2)]{KSS26Multiphase}). Then, 
\begin{equation*}
\sum_{\ell\neq j}u_{*,\ell} - u_{*,j} - U^{C_{0}\chi_{B_{R}}} \quad \text{is subharmonic in $B_{R}$.}
\end{equation*}
Arguing as in \cite[Corollary~3.4]{KSS26Multiphase}, we deduce that $\bfu_{*}$ admits continuous representative on $B_{R}$ for each $R>0$. Since $R>0$ is arbitrary, the conclusion follows. 
\end{proof}

We now establish the existence of minimizers. 

\begin{proof}[Proof of \Cref{thm:existence}] 
Let $\{\bfu^{(h)}\}_{h\in\mN}\subset\mS_{m}$ be any sequence such that $\mF_{\bfV}(\bfu^{(h)})\rightarrow\inf_{\bfu\in\mS_{m}}\mF_{\bfV}(\bfu)$. For each $j=1,\cdots,m$, by using \cite[Theorem~1.10]{KKO26population}, $v_{*,j}$ be the smallest minimizer of $\mF_{V_{j}}$ in $\mK$, which belongs to 
\begin{equation}
\mK_0:=\left\{ u\in \mK : u\le \frac{g_{\alpha}\abs{x}}{g_{\alpha}(R)} \right\} ,\quad g_{\alpha}(r):=\abs{r}^{-\frac{d-2}{2}}K_{\frac{d-2}{2}}(\alpha r), \label{eq:space-K0}
\end{equation}
where $K_{\nu}$ denotes the modified Bessel function of the second kind of order $\nu\in\mR$. Using the properties of Bessel functions (cf. \cite[Appendix~C]{KKO26population}), it satisfies 
\begin{equation*}
v_{*,j} \le C\abs{x}^{-\frac{d-1}{2}}e^{-\alpha\abs{x}} \quad \text{for all $x\in\mR^{d}\setminus\{0\}$} 
\end{equation*}
for some $C=C(R,d,\alpha)$. By \cite[Lemma~3.1(a)]{KKO26population}, we see that 
\begin{equation*}
\mF_{V_{j}}(\min\{u_{j}^{(h)},v_{*,j}\}) \le \mF_{V_{j}}(u_{j}^{(h)}) \quad \text{for all $j=1,\cdots,m$.} 
\end{equation*}
We define 
\begin{equation*}
\tilde{\bfu}^{(h)}=(\tilde{u}_{1}^{(h)},\cdots,\tilde{u}_{m}^{(h)}) ,\quad \tilde{u}_{j}^{(h)} := \min\{u_{j}^{(h)},v_{*,j}\}. 
\end{equation*}
We now see that $\tilde{\bfu}^{(h)}\in\mS_{m}$ for all $h\in\mN$. Thus,  $\{\tilde{\bfu}^{(h)}\}_{h\in\mN}\subset\mS_{m}$ is a sequence such that $\mF_{\bfV}(\tilde{\bfu}^{(h)})\rightarrow\inf_{\bfu\in\mS_{m}}\mF_{\bfV}(\bfu)$ with $\tilde{u}_{j}^{(h)}\in \mK_0$ for all $j=1,\cdots,m$. 
By \cite[Remark~3.3]{KKO26population}, then there exists a subsequence, still denoted by $\{\tilde{\bfu}^{(h)}\}_{h\in\mN}\subset\mS_{m}$, such that $\tilde{\bfu}^{(h)} \rightarrow \bfu_{*}$ strongly in $(L^{2}(\mR^{d}))^{m}$ and weakly in $(H^{1}(\mR^{d}))^{m}$. Consequently, 
\begin{equation*}
\mF_{\bfV}(\bfu_{*}) \le \liminf_{h\rightarrow\infty}\mF_{\bfV}(\tilde{\bfu}^{(h)}). 
\end{equation*}
Using \cite[Theorem~3.12]{Rud87RealComplexAnalysis}, one can choose a further subsequence, still denoted by $\{\tilde{\bfu}^{(h)}\}_{h\in\mN}\subset\mS_{m}$, such that $\tilde{\bfu}^{(h)}(x) \rightarrow \bfu_{*}(x)$ for a.e. $x\in\mR^{d}$, showing that $\bfu_{*}\in\mS_{m}$. 
Finally, the fact $\abs{\bfu_{*}}\not\equiv 0$ follows from \eqref{eq:nontrivial}. 
\end{proof}

We now turn to the proof of \Cref{propr:m-coexist}. 

\begin{proof}[Proof of \Cref{propr:m-coexist}]
For each $r_{j}>1$ to be chosen later, we consider the cut-off function $\breve{\psi}_{r_{j}}\in C_{c}^{\infty}(B_{r_{j}}(0))$ with $0\le \breve{\psi}_{r_{j}}\le 1$ and $\breve{\psi}_{r_{j}}=1$ in $B_{r_{j}-1}(0)$. Let $\breve{v}_{*,j}$ be the smallest minimizer of $\mF_{\breve{V}_{j}}$ in $\mK$. Since both $\breve{v}_{*,j}(x)$ and $\nabla \breve{v}_{*,j}(x)$ decay exponentially as $\abs{x}\rightarrow\infty$, we have 
\begin{equation*}
\norm{\breve{v}_{*,j}\breve{\psi}_{r_{j}}-\breve{v}_{*,j}}_{H^{1}(\mR^{d})} \rightarrow 0 \quad \text{as $r_{j}\rightarrow\infty$,} 
\end{equation*}
and thus 
\begin{equation*}
\mF_{\breve{V}_{j}}(\breve{v}_{*,j}\breve{\psi}_{r_{j}}) \rightarrow \mF_{\breve{V}_{j}}(\breve{v}_{*,j}). 
\end{equation*} 
Therefore, we can choose $r_{1},\cdots,r_{m}>0$, depending only on $d$ and $\breve{V}_{j}$ such that 
\begin{equation*}
\sum_{\ell=1}^{m}\mF_{V_{\ell}}(v_{*,\ell}\psi_{r_{\ell}}) = \sum_{\ell=1}^{m}\mF_{\breve{V}_{\ell}}(\breve{v}_{*,\ell}\breve{\psi}_{r_{\ell}}) < \sum_{\ell\neq j}\mF_{\breve{V}_{\ell}}(\breve{v}_{*,\ell}) = \sum_{\ell\neq j}\mF_{V_{\ell}}(v_{*,\ell}) 
\end{equation*}
with 
\begin{equation*}
v_{*,j}(x) := \breve{v}_{*,j}(x-x^{(j)}) ,\quad \psi_{r_j}(x) := \breve{\psi}_{r_j}(x-x^{(j)}) ,\quad V_{j}(x) := \breve{V}_{j}(x-x^{(j)}) 
\end{equation*} 
for all $j=1,\cdots,m$ and for all $x^{(1)},\cdots,x^{(m)}\in\mR^{d}$. We are now ready to conclude the proof of the theorem under assumption \eqref{eq:mutually-far-away}. Otherwise, suppose that there exists a minimizer $\bfu_{*}$ of \eqref{eq:main-minimizing} in $\mS_{m}$ such that $u_{*,j}\equiv 0$ for some $j=1,\cdots,m$. Then 
\begin{equation*}
\mF_{\bfV}(\bfu_{*}) \ge \sum_{\ell\neq j}\mF_{V_{\ell}}(v_{*,\ell}) > \sum_{\ell=1}^{m}\mF_{V_{\ell}}(v_{*,\ell}\psi_{r_{\ell}}). 
\end{equation*}
On the other hand, condition \eqref{eq:mutually-far-away} ensures that 
\begin{equation*}
(v_{*,1}\psi_{r_{1}},\cdots,v_{*,m}\psi_{r_{m}}) \in \mS_{m}, 
\end{equation*}
which contradicts the minimality of $\bfu_{*}$. 
\end{proof}

Before proving the results concerning the singular limit, we first slightly extend \cite[Lemma~3.1]{KKO26population} in the following lemma. 

\begin{lemma}
Let $L>0$ and let $(L\mK)^{m}$ denote the Cartesian product of $m$ copies of 
\begin{equation*}
L\mK := \left\{Lu : u\in\mK \right\} \equiv\left\{ u\in H^{1}(\mR^{d}) : 0\le u\le L \right\},  
\end{equation*}
where $\mK$ is the space defined in \eqref{eq:constraint-K}, that is 
\begin{equation*}
(L\mK)^{m} := \left\{ \bfu\in (H^{1}(\mR^{d}))^{m} : 0\le u_{i}\le L \right\}. 
\end{equation*}
Let $\bfV^{(1)},\bfV^{(2)}\in(L^{\infty}(\mR^{d}))^{m}$ be potential satisfying $V_{i}^{(1)}\ge V_{i}^{(2)}$ in $\mR^{d}$ for all $i=1,\cdots,m$. For any $\bfu,\bfv\in(L\mK)^{m}$, then 
\begin{equation*}
\begin{aligned} 
\min\{\bfu,\bfv\} &:= \left( \min\{u_{1},v_{1}\},\cdots,\min\{u_{m},v_{m}\} \right) \in (L\mK)^{m}, \\ 
\max\{\bfu,\bfv\} &:= \left( \max\{u_{1},v_{1}\},\cdots,\max\{u_{m},v_{m}\} \right) \in (L\mK)^{m}, 
\end{aligned} 
\end{equation*} 
and the following inequality holds for all $0<\varepsilon\le+\infty$ and $0\le\kappa<+\infty$: 
\begin{equation*}
\calE_{\bfV^{(1)},\varepsilon,\kappa}(\min\{\bfu,\bfv\})+\mF_{\bfV^{(2)}}(\max\{\bfu,\bfv\}) \le \calE_{\bfV^{(1)},\varepsilon,\kappa}(\bfu)+\mF_{\bfV^{(2)}}(\bfv).  
\end{equation*}
In particular, if $\bfv$ minimizes $\mF_{\bfV^{(2)}}$ over $(L\mK)^{m}$, then 
\begin{equation*}
\calE_{\bfV^{(1)},\varepsilon,\kappa}(\min\{\bfu,\bfv\}) \le \calE_{\bfV^{(1)},\varepsilon,\kappa}(\bfu).  
\end{equation*}
\end{lemma}

\begin{proof}
Following exactly the same arguments as in \cite[Lemma~3.1]{KKO26population}, we obtain 
\begin{equation*}
\mF_{\bfV^{(1)}}(\min\{\bfu,\bfv\})+\mF_{\bfV^{(2)}}(\max\{\bfu,\bfv\}) \le \mF_{\bfV^{(1)}}(\bfu)+\mF_{\bfV^{(2)}}(\bfv).  
\end{equation*}
Moreover, 
\begin{equation*}
\int_{\mR^{d}}(\min\{u_{i},v_{i}\}-1)_{+}^{2}\,\rmd\sfm \le \int_{\mR^{d}}(u_{i}-1)_{+}^{2}\,\rmd\sfm \quad \text{for all $i=1,\cdots,m$,}
\end{equation*}
and 
\begin{equation*}
\int_{\mR^{d}}(\min\{u_{i},v_{i}\})^{2}(\min\{u_{j},v_{j}\})^{2}\,\rmd\sfm \le \int_{\mR^{d}}u_{i}^{2}u_{j}^{2}\,\rmd\sfm \quad \text{for all $i,j\in\{1,\cdots,m\}$.}
\end{equation*}
Summing these three inequalities according to \eqref{eq:energy1} yields the desired result. 
\end{proof}

We now prove the results concerning the singular limit. 

\begin{proof}[Proof of \Cref{thm:well-posed}\ref{itm:singular-limit-a}] 
Since the proof is somewhat lengthy, we divide it into several steps for clarity. 

\medskip 

\noindent \emph{Step 1: First, we show that $\inf_{(H_{+}^{1}(\mR^{d}))^{m}}\calE_{\bfV,\varepsilon,\kappa}>-\infty$.}
By \eqref{eq:nonvanishing}, $(V_{i})_{-}=0$ and $(V_{i})_{+}\geq\alpha^{2}$ a.e. in $\mR^d\setminus B_R$. Using the inequality
$s^2\leq2+2(s-1)_+^2$ for $s\geq0$, we have 
\begin{equation} 
\begin{aligned}
 &\frac{1}{2}\int_{B_R} (V_{i})_{-} u_{i}^{2} \,\rmd\sfm 
\leq \frac{\norm{\bfV_{-}}_{L^{\infty}(\mR^{d})}}{2}\int_{B_R}u_i^2\,\rmd\sfm
\leq  \frac{\norm{\bfV_{-}}_{L^{\infty}(\mR^{d})}}{2}\int_{B_R}(2+2(u_i-1)_+^2)\,\rmd\sfm \\
&\leq \norm{\bfV_{-}}_{L^{\infty}(\mR^{d})}\abs{B_R}+\norm{\bfV_{-}}_{L^{\infty}(\mR^{d})}\|(u_i-1)_+\|_{L^2(\mR^{d})}^2.   
\end{aligned} \label{ineq1} 
\end{equation} 
It follows that
\begin{equation*}
\begin{aligned} 
& \frac{1}{2\varepsilon}\|(u_i-1)_+\|_{L^2(\mR^{d})}^2-\frac{1}{2}\int_{B_R}(V_{i})_{-}u_i^2\,\rmd\sfm \\
&\quad \geq -\norm{\bfV_{-}}_{L^{\infty}(\mR^{d})}\abs{B_R}+\Big(\frac{1}{2\varepsilon}-\norm{\bfV_{-}}_{L^{\infty}(\mR^{d})}\Big)\|(u_i-1)_+\|_{L^2(\mR^{d})}^2.  
\end{aligned} 
\end{equation*}
Consequently, for any $\bfu\in(H_{+}^{1}(\mR^{d}))^{m}$, if $\varepsilon \norm{\bfV^{-}}_{L^{\infty}(\mR^{d})}\leq1/4$, then
\begin{equation} 
\begin{aligned}
\calE_{\bfV,\varepsilon,\kappa}(\mathbf u)
&\geq \sum_{i=1}^{m}\|\nabla u_i\|_{L^2(\mR^{d})}^2
+\alpha^{2}\sum_{i=1}^{m}\|u_i\|_{L^2(\mathbb R^d\setminus B_R)}^{2}
+\frac{1}{2\varepsilon}\sum_{i=1}^{m}\|(u_i-1)_+\|_{L^2(\mR^{d})}^{2} \\
&\quad + \kappa\sum_{1\le i<j\le m}\|u_{i}u_{j}\|_{L^2(\mR^{d})}^{2} 
-2m\norm{\bfV_{-}}_{L^{\infty}(\mR^{d})}\abs{B_R}.
\end{aligned} \label{coercivity1} 
\end{equation} 

\medskip 

\noindent \emph{Step 2: Next, we show uniform boundedness of the zero-sublevel sets.} 
Fix $\overline{\varepsilon}\in(0,(4\norm{\bfV_{-}}_{L^{\infty}(\mR^{d})})^{-1})$. For each $\varepsilon\in(0,\overline{\varepsilon})$, we consider the sublevel set 
\begin{equation}
\left\{\bfu\in (H_{+}^{1}(\mR^{d}))^{m}: \calE_{\bfV,\varepsilon,\kappa}(\bfu)\leq 0\right\} \label{eq:sublevel-set}
\end{equation}
By \eqref{coercivity1}, every element $\bfu$ of the sublevel set \eqref{eq:sublevel-set} satisfies 
\begin{equation*}
\sum_{i=1}^{m}\|\nabla u_i\|_{L^2(\mR^{d})}^2\leq2C_0,\quad
\sum_{i=1}^{m}\|u_i\|_{L^2(\mathbb{R}^d\setminus B_R)}^2\leq\frac{2C_0}{{\alpha}^2},
\quad 
\sum_{i=1}^{m}\|(u_i-1)_+\|_{L^2(\mR^{d})}^2
\leq4\varepsilon C_0, 
\end{equation*}
where $C_0:=m\norm{\bfV_{-}}_{L^{\infty}(\mR^{d})}\abs{B_R}\geq 0$. 
On the other hand, from \eqref{ineq1}, we obtain
\begin{equation*} 
\sum_{i=1}^m\|u_i\|_{L^2(B_R)}^2
\leq 2m\abs{B_R}+8\bar\varepsilon C_0.
\end{equation*} 
Consequently, we have
\[
\sum_{i=1}^{m}\|u_i\|_{H^1(\mathbb{R}^d)}^2
\leq
\overbrace{2C_0+\frac{2C_0}{{\alpha}^2}+2m\abs{B_R}+8\bar\varepsilon C_0}^{\text{independent of $\varepsilon\in(0,\overline{\varepsilon})$ and $\kappa>0$}},
\]
which proves the uniform boundedness.

\medskip 

\noindent \emph{Step 3: We now prove the existence of a global minimizer in $(H_{+}^{1}(\mR^{d}))^{m}$ by constructing a suitable minimizing sequence by truncation.} 
Let $\{\tilde{\bfu}^{(h)}\}_{h\in\mN}$ be any minimizing sequence for $\calE_{\bfV,\varepsilon,\kappa}$ over $(H_{+}^{1}(\mR^{d}))^{m}$. For each $\varepsilon\in(0,\overline{\varepsilon})$, we set 
\begin{equation*}
L_{\varepsilon} := \frac{1}{1-\varepsilon\norm{\bfV_{-}}_{L^{\infty}(\mR^{d})}} \ge 1, 
\end{equation*}
and the mapping 
\begin{equation*}
G_{i}(x,s) := V_{i}(x)s^{2} + \frac{1}{\varepsilon}(s-1)_{+}^{2} \quad \text{for all $i=1,\cdots,m$,}
\end{equation*}
so that 
\begin{equation*}
\calE_{\bfV,\varepsilon,\kappa}(\bfu) = \sum_{i=1}^{m}\int_{\mR^{d}} \left(\abs{\nabla u_{i}(x)}^{2} + G_{i}(x,u_{i})\right)\,\rmd\sfm + \kappa\sum_{1\le i<j\le m}\int_{\mR^{d}}u_{i}^{2}u_{j}^{2}\,\rmd\sfm. 
\end{equation*}
We now define the sequence $\{\breve{\bfu}^{(h)}\}_{h\in\mN}\subset(L_{\varepsilon}\mK)^{m}$ by 
\begin{equation*}
\breve{u}_{i}^{(h)} := \min\{\tilde{u}_{i}^{(h)},L_{\varepsilon}\} \quad \text{for all $i=1,\cdots,m$.} 
\end{equation*}
It is not difficult to see that 
\begin{equation}
\int_{\mR^{d}}\abs{\nabla \breve{u}_{i}^{(h)}}^{2}\,\rmd\sfm \le \int_{\mR^{d}}\abs{\nabla \tilde{u}_{i}^{(h)}}^{2}\,\rmd\sfm ,\quad \int_{\mR^{d}}(\breve{u}_{i}^{(h)})^{2}(\breve{u}_{j}^{(h)})^{2}\,\rmd\sfm \le \int_{\mR^{d}}(\tilde{u}_{i}^{(h)})^{2}(\tilde{u}_{j}^{(h)})^{2}\,\rmd\sfm \label{eq:truncate-minimizing-sequence1}
\end{equation}
for all $i,j\in\{1,\cdots,m\}$. 
Since, for $s> L_{\varepsilon}$, we have 
\begin{equation}
\frac{1}{2}\frac{\partial}{\partial s}G_{i}(x,s)\ge -\norm{\bfV_{-}}_{L^{\infty}(\mR^{d})}s + \frac{s-1}{\varepsilon} = \frac{(1-\varepsilon\norm{\bfV_{-}}_{L^{\infty}(\mR^{d})})s-1}{\varepsilon} > 0 \label{eq:G-strict-increasing}
\end{equation}
for all $x\in\mR^{d}$, we see that 
\begin{equation}
\int_{\mR^{d}} G_{i}(x,\breve{u}_{i}^{(h)})\,\rmd\sfm \le \int_{\mR^{d}} G_{i}(x,\tilde{u}_{i}^{(h)})\,\rmd\sfm. \label{eq:truncate-minimizing-sequence2}
\end{equation}
Combining \eqref{eq:truncate-minimizing-sequence1} and \eqref{eq:truncate-minimizing-sequence2}, we see that 
\begin{equation*}
\calE_{\bfV,\varepsilon,\kappa}(\breve{\bfu}^{(h)}) \le \calE_{\bfV,\varepsilon,\kappa}(\tilde{\bfu}^{(h)}), 
\end{equation*}
which shows that $\{\breve{\bfu}^{(h)}\}_{h\in\mN}\subset(L_{\varepsilon}\mK)^{m}$ is also a minimizing sequence for $\calE_{\bfV,\varepsilon,\kappa}$ over $(H_{+}^{1}(\mR^{d}))^{m}$. 

We consider the potential $V_{*}\in L^{\infty}(\mR^{d})$ defined by 
\begin{equation*}
V_{*}(x)=\alpha^{2} \text{ for all $x\in \mR^{d}\setminus B_{R}$} ,\quad V_{*}(x)=-\tilde{\beta}^{2}\text{ for all $x\in B_{R}$} 
\end{equation*}
for any $\tilde{\beta}^{2}>\max\{\lambda^{*}(B_{R}),\norm{\bfV_{-}}_{L^{\infty}(\mR^{d})}\}$. Then there exists a unique nontrivial global minimizer $v_{\varepsilon}$ of $\mF_{V_{*}}$ over $L_{\varepsilon}\mK$, which is necessarily of the form 
\begin{equation*}
v_{\varepsilon}(x)=\left\{\begin{aligned}
&L_{\varepsilon} && \text{for all $x\in B_{R_{*}}$,} \\ 
&a_{\varepsilon}(\tilde{\beta}\abs{x})^{-\frac{d-2}{2}}J_{\frac{d-2}{2}}(\tilde{\beta}\abs{x}) + b_{\varepsilon}(\tilde{\beta}\abs{x})^{-\frac{d-2}{2}}Y_{\frac{d-2}{2}}(\tilde{\beta}\abs{x}) && \text{for all $x\in B_{R}\setminus B_{R_{*}}$,} \\ 
& c_{\varepsilon}(\alpha\abs{x})^{-\frac{d-2}{2}}K_{\frac{d-2}{2}}(\alpha\abs{x}) && \text{for all $x\in\mR^{d}\setminus B_{R}$,} 
\end{aligned}\right. 
\end{equation*}
for some constants $a_{\varepsilon},b_{\varepsilon},c_{\varepsilon}$. It is not difficult to see that 
\begin{equation}
v_{\varepsilon} \in L_{\varepsilon}\mK_{0}:=\left\{L_{\varepsilon}v : v\in\mK_{0}\right\}, \label{eq:truncation}
\end{equation}
where $\mK_{0}$ is the space given in \eqref{eq:space-K0}. 
Choosing 
\begin{equation*}
\bfV^{(1)}=\bfV ,\quad \bfV^{(2)}=(V_{*},\cdots,V_{*}) ,\quad \bfu=\breve{\bfu}^{(h)} ,\quad \bfv=(v_{\varepsilon},\cdots,v_{\varepsilon})
\end{equation*}
yields 
\begin{equation*}
\calE_{\bfV,\varepsilon,\kappa}(\bfu^{(h)}) \le \calE_{\bfV,\varepsilon,\kappa}(\breve{\bfu}^{(h)}), \quad \bfu^{(h)}:=\min\{\breve{\bfu}^{(h)},(v_{\varepsilon},\cdots,v_{\varepsilon})\}
\end{equation*}
which shows that $\{\bfu^{(h)}\}_{h\in\mN}\subset(L_{\varepsilon}\mK_{0})^{m}$ is also a minimizing sequence for $\calE_{\bfV,\varepsilon,\kappa}$ over $(H_{+}^{1}(\mR^{d}))^{m}$. 

Similar to the proof of \Cref{thm:existence}, by \cite[Remark~3.3]{KKO26population} and \cite[Theorem~3.12]{Rud87RealComplexAnalysis}, there exists a subsequence, still denoted by $\{\bfu^{(h)}\}_{h\in\mN}\subset(L_{\varepsilon}\mK_{0})^{m}$, such that $\bfu^{(h)}\rightarrow\bfu_{(\varepsilon,\kappa)}$ strongly in $(L^{2}(\mR^{d}))^{m}$, weakly in $(H^{1}(\mR^{d}))$, and almost everywhere in $\mR^{d}$. In particular, $\bfu_{(\varepsilon,\kappa)}\in(L_{\varepsilon}\mK_{0})^{m}$. We now prove that $\bfu_{(\varepsilon,\kappa)}=(u_{(\varepsilon,\kappa),1},\cdots,u_{(\varepsilon,\kappa),m})$ attains the infimum. Since, in any normed space $\mX$, the norm $\norm{\cdot}_{\mX}$ is lower semicontinuous with respect to the weak topology of $\mX$, see, e.g. \cite[Corollary~3.9]{Bre11PDE}, we obtain 
\begin{equation*}
\mF_{\bfV}(\bfu_{(\varepsilon,\kappa)}) \le \liminf_{h\rightarrow\infty} \mF_{\bfV}(\bfu^{(h)}). 
\end{equation*}
Finally, in view of \eqref{eq:energy1}, Fatou's lemma yields 
\begin{equation*} 
\calE_{\bfV,\varepsilon,\kappa}(\bfu_{(\varepsilon,\kappa)})
\le
\liminf_{h\to\infty}
\calE_{\bfV,\varepsilon,\kappa} (\bfu^{(h)}) = \inf_{(H_{+}^{1}(\mR^{d}))^{m}}\calE_{\bfV,\varepsilon,\kappa}, 
\end{equation*}
hence $\bfu_{(\varepsilon,\kappa)}\in(L_{\varepsilon}\mK_{0})^{m}$ is a global minimizer of $\calE_{\bfV,\varepsilon,\kappa}$ over $(H_{+}^{1}(\mR^{d}))^{m}$. 
\end{proof}


\begin{remark}[An alternative proof of \emph{Step~3}]\label{rem:LSC}
If one only seeks to establish the existence of minimizers, without requiring pointwise bounds, an alternative proof can be given using a diagonal argument to construct a sequence that converges strongly in $(L_{\mathrm{loc}}^{2}(\mR^{d}))^{m}$. In contrast, \emph{Step~3} above can be viewed as a constructive proof, yielding a sequence that satisfies an exponentially decaying pointwise bound and converges strongly even in $(L^{2}(\mR^{d}))^{m}$. In fact, the entirety of \emph{Step~3} can be replaced by \Cref{rem:LSC} without affecting our main results, since the pointwise bound is established separately in \Cref{thm:well-posed}\ref{itm:singular-limit-b}. 

Let $\{\bfu^{(h)}\}_{h\in\mN}$ be any minimizing sequence for $\calE_{\bfV,\varepsilon,\kappa}$ over $(H_{+}^{1}(\mR^{d}))^{m}$. 
Since $\calE_{\bfV,\varepsilon,\kappa}(\mathbf{0})=0$, we may extract a subsequence, still denoted by $\{\bfu^{(h)}\}_{h\in\mN}$, that belongs to the sublevel set \eqref{eq:sublevel-set}. By the Banach-Alaoglu theorem, one can then extract a further subsequence, again denoted by $\{\bfu^{(h)}\}_{h\in\mN}$, such that 
\begin{equation*}
\bfu^{(h)} \rightarrow \bfu \text{ in $(H^{1}(\mR^{d}))^{m}$ weakly.}
\end{equation*} 

Recall that, by the Rellich-Kondrachov theorem, the local embedding $H^{1}(B_{r})\hookrightarrow L^{2}(B_{r})$ is compact for every $r>0$.  
Taking $r=1$, we can extract a subsequence $\left\{\bfu^{(h,1)}\right\}_{h\in\mN}$ of $\left\{\bfu^{(h)}\right\}_{h\in\mN}$ that converges strongly in $L^{2}(B_{1})$. Inductively, for each $h'\in\mN$, we can extract a subsequence $\left\{\bfu^{(h,h'+1)}\right\}_{h\in\mN}$ of $\left\{\bfu^{(h,h')}\right\}_{h\in\mN}$ that converges strongly in $L^{2}(B_{h'})$. 
Thus, the `diagonal sequence' $\left\{\bfu^{(h,h)}\right\}_{h\in\mN}$ converges strongly in $(L^{2}(B_{r}))^{m}$ for every $r>0$.
In particular, by \cite[Theorem~3.12]{Rud87RealComplexAnalysis}, after passing to a further subsequence, still denoted  by $\left\{\bfu^{(h,h)}\right\}_{h\in\mN}\subset(H_{+}^{1}(\mR^{d}))^{m}$, we may assume that it converges to some limit $\bfu_{(\varepsilon,\kappa)}\in(H_{+}^{1}(\mR^{d}))^{m}$ a.e. in $\mR^{d}$. 

We now prove that $\bfu_{(\varepsilon,\kappa)}=(u_{(\varepsilon,\kappa),1},\cdots,u_{(\varepsilon,\kappa),m})$ attains the infimum. Recall that, in any normed space $\mX$, the norm $\norm{\cdot}_{\mX}$ is lower semicontinuous with respect to the weak topology of $\mX$, see, e.g. \cite[Corollary~3.9]{Bre11PDE}. By assumption \eqref{eq:nonvanishing}, the mapping 
\begin{equation*}
\bfu\in(H^{1}(\mR^{d}))^{m} \mapsto \left(\sum_{i=1}^m\int_{\mR^d}
\left(\abs{\nabla u_{i}}^{2} + (V_{i})_{+} u_{i}^{2}\right)\,\rmd\sfm + \norm{\bfu}_{L^{2}(B_{R})}^{2}\right)^{1/2} 
\end{equation*}
defines a norm on $(H^{1}(\mR^{d}))^{m}$. Therefore, the weak lower semicontinuity of the norm yields
\begin{equation*}
\begin{aligned}
& \sum_{i=1}^m\int_{\mR^d}
\left(\abs{\nabla u_{(\varepsilon,\kappa),i}}^{2} + (V_{i})_{+} u_{(\varepsilon,\kappa),i}^{2}\right)\,\rmd\sfm + \norm{\bfu_{(\varepsilon,\kappa)}}_{L^{2}(B_{R})}^{2} \\ 
&\quad \le \liminf_{h\to\infty}
\left( \sum_{i=1}^m\int_{\mR^d}
\left(\abs{\nabla u_{i}^{(h,h)}}^{2} + (V_{i})_{+} (u_{i}^{(h,h)})^{2} \right)\,\rmd\sfm + \norm{\bfu^{(h,h)}}_{L^{2}(B_{R})}^{2}\right). 
\end{aligned}
\end{equation*}
Using $\bfu^{(h,h)}\to\bfu_{(\varepsilon,\kappa)}$ strongly in $(L^2(B_R))^{m}$, we obtain 
\begin{equation}
\begin{aligned} 
& \sum_{i=1}^m\int_{\mR^d}
\left(\abs{\nabla u_{(\varepsilon,\kappa),i}}^{2} + (V_{i})_{+} u_{(\varepsilon,\kappa),i}^{2}\right)\,\rmd\sfm \\ 
&\quad \leq\liminf_{h\to\infty}
\sum_{i=1}^m\int_{\mR^d}
\left(\abs{\nabla u_{i}^{(h,h)}}^{2} + (V_{i})_{+} (u_{i}^{(h,h)})^{2} \right)\,\rmd\sfm. 
\end{aligned}\label{eq:lsc-mF}
\end{equation}
and 
\begin{equation*}
\int_{\mR^d}(V_{i})_{-} (u_{i}^{(h,h)})^{2}\,\rmd\sfm \rightarrow \int_{\mR^d}(V_{i})_{-}u_{(\varepsilon,\kappa),i}^{2} \,\rmd\sfm.  
\end{equation*}
since $\supp\,((V_{i})_{-})\subset \overline{B_{R}}$. 
Therefore, 
\begin{equation*} 
\calE_{\bfV,\varepsilon,\kappa}(\bfu_{(\varepsilon,\kappa)})
\le
\liminf_{h\to\infty}
\calE_{\bfV,\varepsilon,\kappa} (\bfu^{(h,h)}) = \inf_{(H_{+}^{1}(\mR^{d}))^{m}}\calE_{\bfV,\varepsilon,\kappa}, 
\end{equation*}
and hence $\bfu_{(\varepsilon,\kappa)}$ is a global minimizer of $\calE_{\bfV,\varepsilon,\kappa}$ over $(H_{+}^{1}(\mR^{d}))^{m}$. 
\end{remark}

We have already constructed a minimizer that satisfies \eqref{eq:pointwise-bound}. We now show that this bound is, in fact, necessary for every minimizer, thus proving \Cref{thm:well-posed}\ref{itm:singular-limit-b}. 

\begin{proof}[Proof of \Cref{thm:well-posed}\ref{itm:singular-limit-b}]
Let $\bfu_{(\varepsilon,\kappa)}$ be any global minimizer of $\calE_{\bfV,\varepsilon,\kappa}$ over $(H_{+}^{1}(\mR^{d}))^{m}$. 
We first show that 
\begin{equation}
\max_{i=1,\cdots,m} u_{(\varepsilon,\kappa),i}(x) \le L_{\varepsilon}=\frac{1}{1-\varepsilon\norm{\bfV_{-}}_{L^{\infty}(\mR^{d})}} \label{eq:pointwise-bound-1}
\end{equation} 
for a.e. $x\in\mR^{d}$. Suppose, to the contrary, that \eqref{eq:pointwise-bound-1} fails. Then there exists $i=1,\cdots,m$ such that $\left\{u_{(\varepsilon,\kappa),i}>L_{\varepsilon}\right\}$ has positive measure. In this case, by \eqref{eq:G-strict-increasing}, we have the strict inequality 
\begin{equation*}
\int_{\mR^{d}} G_{i}(x,\min\{u_{(\varepsilon,\kappa),i},L_{\varepsilon}\})\,\rmd\sfm < \int_{\mR^{d}} G_{i}(x,u_{(\varepsilon,\kappa),i})\,\rmd\sfm. 
\end{equation*} 
By an argument similar to that used in the derivation of \eqref{eq:truncate-minimizing-sequence1} and \eqref{eq:truncate-minimizing-sequence2}, we see that 
\begin{equation*}
\calE_{\bfV,\varepsilon,\kappa}(u_{(\varepsilon,\kappa),1},\cdots,u_{(\varepsilon,\kappa),(i-1)},\min\{u_{(\varepsilon,\kappa),i},L_{\varepsilon}\},u_{(\varepsilon,\kappa),(i+1)},u_{(\varepsilon,\kappa),m}) < \calE_{\bfV,\varepsilon,\kappa}(\bfu_{(\varepsilon,\kappa)}), 
\end{equation*}
which contradicts the minimality of $\bfu_{(\varepsilon,\kappa)}$.

We next prove \eqref{eq:EL}. As noted in \Cref{rem:even-extension}, $\bfu_{(\varepsilon,\kappa)}$ is also a minimizer of the even extension 
\begin{equation*}
\begin{aligned}
\tilde{\calE}_{\bfV,\varepsilon,\kappa}(\bfu) &:= \calE_{\bfV,\varepsilon,\kappa}(\abs{u_{1}},\cdots,\abs{u_{m}}) \\ 
&= \mF_{\bfV}(\bfu) +\frac{1}{\varepsilon}\sum_{i=1}^{m} \int_{\mR^d}(\abs{u_{i}}-1)_{+}^{2}\,\rmd\sfm + \kappa \sum_{1\leq i<j\leq m} \int_{\mR^d}u_{i}^{2}u_{j}^{2}\, \rmd\sfm
\end{aligned}
\end{equation*}
over $(H^{1}(\mR^{d}))^{m}$. Fix any nonnegative function $\phi\in C_{c}^{\infty}(\mR^{d})$. For each $i=1,\cdots,m$, we have $\bfu_{(\varepsilon,\kappa)}\pm\delta\phi\mathbf{e}_{i} \in (H^{1}(\mR^{d}))^{m}$ for all $\delta>0$, where $\mathbf{e}_{i}$ denotes the $i^{\rm th}$ column of the $m\times m$ matrix. By the minimality of $\bfu_{(\varepsilon,\kappa)}$,  
\allowdisplaybreaks
\begin{align*}
0 &\le \tilde{\calE}_{\bfV,\varepsilon,\kappa}(\bfu_{(\varepsilon,\kappa)}\pm\delta\phi\mathbf{e}_{i}) - \tilde{\calE}_{\bfV,\varepsilon,\kappa}(\bfu_{(\varepsilon,\kappa)}) \\  
&= \pm2\delta\int_{\mR^{d}} \left(\nabla u_{(\varepsilon,\kappa),i}\cdot\nabla \phi + V_{i}u_{(\varepsilon,\kappa),i}\phi\right) \,\rmd\sfm + \overbrace{\delta^{2}\int_{\mR^{d}}(\abs{\nabla\phi}^{2}+V_{i}\phi^{2})\,\rmd\sfm}^{=O(\delta^{2})} \\ 
&\qquad + \frac{1}{\varepsilon}\int_{\mR^{d}} \left( (\abs{u_{(\varepsilon,\kappa),i}\pm\delta\phi}-1)_{+}^{2} - (\abs{u_{(\varepsilon,\kappa),i}}-1)_{+}^{2} \right)\,\rmd\sfm \\
&\qquad \pm2\delta\kappa\sum_{j\neq i}\int_{\mR^{d}} u_{(\varepsilon,\kappa),i}u_{(\varepsilon,\kappa),j}^{2}\phi\,\rmd\sfm + \overbrace{\delta^{2}\kappa\sum_{j\neq i}\int_{\mR^{d}}\phi^{2}u_{(\varepsilon,\kappa),j}^{2}\,\rmd\sfm}^{\text{$=O(\delta^{2})$, because \eqref{eq:pointwise-bound-1}}} \\ 
&\le \pm2\delta\int_{\mR^{d}} \left(\nabla u_{(\varepsilon,\kappa),i}\cdot\nabla \phi + V_{i}u_{(\varepsilon,\kappa),i}\phi\right) \,\rmd\sfm \\ 
&\quad \pm2\delta \frac{1}{\varepsilon}\int_{\{\abs{u_{(\varepsilon,\kappa),i}\pm\delta\phi}>1\}\cap\{\abs{u_{(\varepsilon,\kappa),i}}>1\}} (u_{(\varepsilon,\kappa),i}-1)\phi\,\rmd\sfm \\ 
&\quad + \frac{1}{\varepsilon} \overbrace{\int_{\{\abs{u_{(\varepsilon,\kappa),i}\pm\delta\phi}>1\}\cap\{\abs{u_{(\varepsilon,\kappa),i}}\le1\}}  (\abs{u_{(\varepsilon,\kappa),i}\pm\delta\phi}-1)^{2}\,\rmd\sfm}^{\text{$=O(\delta^{2})$, because $0<\abs{u_{(\varepsilon,\kappa),i}\pm\delta\phi}-1\le\abs{u_{(\varepsilon,\kappa),i}}-1+\delta\abs{\phi}\le\delta\abs{\phi}$}} \\
&\quad \pm2\delta\kappa\sum_{j\neq i}\int_{\mR^{d}} u_{(\varepsilon,\kappa),i}u_{(\varepsilon,\kappa),j}^{2}\phi\,\rmd\sfm + O(\delta^{2}). 
\end{align*}
Dividing by $2\delta$ and letting $\delta\searrow0$, we obtain 
\begin{equation*}
\pm\left(\Delta u_{(\varepsilon,\kappa),i} - V_{i}u_{(\varepsilon,\kappa),i}\right) \le \pm\left(\frac{1}{\varepsilon}(u_{(\varepsilon,\kappa),i}-1)_{+} +\kappa u_{(\varepsilon,\kappa),i}\sum_{j\neq i}u_{j}^{2}\right) \quad \text{in $\mR^{d}$,} 
\end{equation*}
we conclude \eqref{eq:EL}.

It remains to prove $\bfu_{(\varepsilon,\kappa)}\in (L_{\varepsilon}\mK_{0})^{m}$. 
Define 
\begin{equation*}
\Psi:=L_{\varepsilon}\frac{g_{\alpha}(\abs{x})}{g_{\alpha}(R)} ,\quad g_{\alpha}(r):=r^{-\frac{d-2}{2}}K_{\frac{d-2}{2}}(\alpha r) \quad \text{for all $r>0$.} 
\end{equation*} 
For each $i=1,\cdots,m$, we see that $\Psi-u_{(\varepsilon,\kappa),i}\in H^{1}(\mR^{d}\setminus\overline{B_{R}})$ satisfies 
\begin{equation*}
(\Delta-\alpha^{2})(\Psi-u_{(\varepsilon,\kappa),i}) = -\overbrace{(\Delta u_{(\varepsilon,\kappa),i}
-V_{i}u_{(\varepsilon,\kappa),i})}^{\ge0\text{ by \eqref{eq:EL}}}-\overbrace{(V_{i}-\alpha^{2})u_{(\varepsilon,\kappa),i}}^{\ge0 \text{ by \eqref{eq:nonvanishing}}} \le 0 \quad \text{ in $\mR^{d}\setminus\overline{B_{R}}$,} 
\end{equation*}
and $(\Psi-u_{(\varepsilon,\kappa),i})|_{\partial B_{R}}=(L_{\varepsilon}-u_{(\varepsilon,\kappa),i})|_{\partial B_{R}}\overset{\eqref{eq:pointwise-bound-1}}{\ge}0$. By the comparison principle (cf \cite[Lemma~A.1]{KKO26population})\footnote{The key point in the proof is that $\Psi-u_{(\varepsilon,\kappa),i}\in H^{1}(\mR^{d}\setminus\overline{B_{R}})$ and $(\Psi-u_{(\varepsilon,\kappa),i})|_{\partial B_{R}}\ge0$ together imply $(\Psi-u_{(\varepsilon,\kappa),i})_{-}\in H_{0}^{1}(\mR^{d}\setminus\overline{B_{R}})$.}, we have $\Psi-u_{(\varepsilon,\kappa),i}\ge0$ in $\mR^{d}\setminus\overline{B_{R}}$ for all $i=1,\cdots,m$. 
Finally, combining this inequality with \eqref{eq:pointwise-bound-1}, we conclude that $\bfu_{(\varepsilon,\kappa)}\in (L_{\varepsilon}\mK_{0})^{m}$. 
\end{proof} 

We now show that the minimization problem \eqref{eq:main-minimizing} arises as the variational limit of energies with increasingly strong crowding and interspecific competition penalties. 

\begin{proof}[Proof of \Cref{thm:singular-limit}\ref{itm:singular-limit-c}] 
In view of \Cref{rem:property-b}, it remains to prove property \ref{eq:Gamma-1}, following the strategy outlined in \Cref{rem:property-a-hint}. If 
\begin{equation*}
\liminf_{n\rightarrow\infty}\calE_{\bfV,\varepsilon_{n},\kappa_{n}}(\bfw^{(n)}) = +\infty, 
\end{equation*}
then there is nothing to prove. We therefore consider the case 
\begin{equation*}
\liminf_{n\rightarrow\infty}\calE_{\bfV,\varepsilon_{n},\kappa_{n}}(\bfw^{(n)}) < +\infty. 
\end{equation*}
From \eqref{coercivity1}, we have 
\begin{equation*} 
\begin{aligned}
& \calE_{\bfV,\varepsilon_{n},\kappa_{n}}(\bfw^{(n)}) + 2m\norm{\bfV_{-}}_{L^{\infty}(\mR^{d})}\abs{B_R} \\ 
& \quad \ge \frac{1}{2\varepsilon_{n}}\sum_{i=1}^{m}\|(w_{i}^{(n)}-1)_+\|_{L^2(\mR^{d})}^{2} +\kappa_{n} \sum_{1\le i<j\le m}\|w_{i}^{(n)}w_{j}^{(n)}\|_{L^2(\mR^{d})}^{2}, 
\end{aligned}  
\end{equation*} 
which implies 
\begin{equation}
\begin{aligned}
\liminf_{n\rightarrow\infty}\sum_{i=1}^{m}\|(w_{i}^{(n)}-1)_+\|_{L^2(\mR^{d})}^{2} &\le 2 \liminf_{n\rightarrow\infty} \left[\varepsilon_{n}\scalebox{0.8}{$\left(\calE_{\bfV,\varepsilon_{n},\kappa_{n}}(\bfw^{(n)}) + 2m\norm{\bfV_{-}}_{L^{\infty}(\mR^{d})}\abs{B_R}\right)$}\right] = 0, \\ 
\liminf_{n\rightarrow\infty}\sum_{1\le i<j\le m}\|w_{i}^{(n)}w_{j}^{(n)}\|_{L^2(\mR^{d})}^{2} &\le \liminf_{n\rightarrow\infty}\left[ \frac{1}{\kappa_{n}}\scalebox{0.8}{$\left(\calE_{\bfV,\varepsilon_{n},\kappa_{n}}(\bfw^{(n)}) + 2m\norm{\bfV_{-}}_{L^{\infty}(\mR^{d})}\abs{B_R}\right)$} \right]=0. 
\end{aligned} \label{eq:hint-done}
\end{equation}
Using the diagonal argument in \Cref{rem:LSC}, one can extract a subsequence, still denoted by $\left\{\bfw^{(n)}\right\}_{n\in\mN}$, such that 
\begin{equation*}
\bfw^{(n)}\rightarrow\bfw \quad \text{a.e. in $\mR^{d}$.} 
\end{equation*}
Therefore, applying Fatou's lemma to \eqref{eq:hint-done}, we obtain 
\begin{equation*}
(w_{i}^{(n)}-1)_{+}\rightarrow0 ,\quad w_{i}^{(n)}w_{j}^{(n)}\rightarrow 0 \quad \text{a.e. in $\mR^{d}$,} 
\end{equation*}
which implies that $\bfw\in\mS_{m}$. Finally, adapting the argument used to prove \eqref{eq:lsc-mF}, we conclude that 
\begin{equation*}
\liminf_{n\rightarrow\infty}\calE_{\bfV,\varepsilon_{n},\kappa_{n}}(\bfw^{(n)}) \ge \liminf_{n\rightarrow\infty}\mF_{\bfV}(\bfw^{(n)}) \ge \mF_{\bfV}(\bfw) = \calE_{\bfV,0,+\infty}(\bfw), 
\end{equation*}
which completes the proof. 
\end{proof}

\begin{proof}[Proof of \Cref{thm:singular-limit}\ref{itm:singular-limit-d}]
By \Cref{thm:well-posed}\ref{itm:singular-limit-a}, there exists $n_{0}\in\mN$ such that, for each $n\ge n_{0}$, the functional $\calE_{\bfV,\varepsilon_{n},\kappa_{n}}$ admits a global minimizer $\bfu^{(n)}$ over $(H_{+}^{1}(\mR^{d}))^{m}$. Consequently, by \Cref{thm:well-posed}\ref{itm:singular-limit-b}, $\bfu^{(n)}$ has a representative in $C_{\rm loc}^{1}(\mR^{d})$ and there exists a constant $C_{R,d,\alpha}>0$, which is independent of $\varepsilon_{n}$ and $\kappa_{n}$, such that 
\begin{equation*}
\max_{i=1,\cdots,m} u_{i}^{(n)}(x) \le 2 \min\left\{1,C_{R,d,\alpha}(1+\abs{x})^{-\frac{d-1}{2}}e^{-\alpha\abs{x}}\right\} \quad \text{for all $x\in\mR^{d}$}
\end{equation*}
and for all $n\ge n_{0}$. On the other hand, applying \eqref{coercivity1} to $\bfu^{(n)}$ gives 
\begin{equation*} 
\sum_{i=1}^{m} \|\nabla u_{i}^{(n)}\|_{L^2(\mR^{d})}^2 \le \overbrace{\calE_{\bfV,\varepsilon_{n},\kappa_{n}}(\bfu^{(n)})}^{\le\calE_{\bfV,\varepsilon_{n},\kappa_{n}}(\mathbf{0})=0} + 2m\norm{\bfV_{-}}_{L^{\infty}(\mR^{d})}\abs{B_R}. 
\end{equation*} 
By the above two inequalities, using the compactness argument that has been used in the proof of \Cref{thm:existence}, there exists a subsequence, still denoted by $\left\{\bfu^{(n)}\right\}_{n\in\mN}$, such that $\bfu^{(n)}\rightarrow\bfu_{*}$ weakly in $(H^{1}(\mR^{d}))^{m}$, strongly in $(L^{2}(\mR^{d}))^{m}$, and a.e. in $\mR^{d}$. The uniform boundedness of 
\begin{equation*}
\kappa_{n}\sum_{1\le i<j\le m}\norm{u_{i}^{(n)}u_{j}^{(n)}}_{L^{2}(\mR^{d})}^{2}, 
\end{equation*}
which is a consequence of \eqref{coercivity1}, implies that $\bfu_{*}\in\mS_{m}$. 

Now, by \Cref{thm:singular-limit}\ref{itm:singular-limit-c}, we have 
\begin{equation*}
\mF_{\bfV}(\bfu_{*}) \le \liminf_{n\rightarrow\infty}\calE_{\bfV,\varepsilon_{n},\kappa_{n}}(\bfu^{(n)}) \le \liminf_{n\rightarrow\infty} \calE_{\bfV,\varepsilon_{n},\kappa_{n}}(\bfv) = \mF_{\bfV}(\bfv) \quad \text{for all $\bfv\in\mS_{m}$.} 
\end{equation*}
Taking $\bfv=\bfu_{*}$ in the above inequality, we obtain 
\begin{equation*}
\liminf_{n\rightarrow\infty}\calE_{\bfV,\varepsilon_{n},\kappa_{n}}(\bfu^{(n)}) = \mF_{\bfV}(\bfu_{*}). 
\end{equation*}
Therefore, we can extract a further subsequence, still denoted by $\bfu^{(n)}$, that satisfies \eqref{eq:minimizing-sequence}. 

Using \eqref{eq:minimizing-sequence} and \eqref{eq:hint-done}, as well as the strong convergence in $(L^{2}(\mR^{d}))^{m}$, we see that, for each $i=1,\cdots,m$, 
\begin{equation*}
\lim_{n\rightarrow\infty}\norm{\nabla u_{i}^{(n)}}_{L^{2}(\mR^{d})}=\norm{\nabla u_{*,i}}_{L^{2}(\mR^{d})}. 
\end{equation*}
Finally, using the fact that weak convergence together with convergence of norms in a Hilbert space implies strong convergence (a special case of a general property of uniformly convex spaces, see \cite[Proposition~3.32]{Bre11PDE}), we conclude that $\bfu^{(n)}\to \bfu_{*}$ strongly in $(H^{1}(\mR^d))^{m}$. 
\end{proof}

\end{sloppypar}

\bibliographystyle{custom}
\bibliography{ref}
\end{document}